\documentclass{article}
\usepackage{fancyhdr}
\usepackage{amscd}
\usepackage{graphicx}
\usepackage{amsmath}
\usepackage{amssymb}
\usepackage{amsthm}
\usepackage{mathrsfs}
\newtheorem{dfn}{Definition}[section]
\newtheorem{thm}[dfn]{Theorem}
\newtheorem{prop}[dfn]{Proposition}
\newtheorem{lem}[dfn]{Lemma}

\newtheorem{rem}[dfn]{Remark}

\newtheorem{ex}[dfn]{Example}

\newtheorem{alg}{Algorithm}

\newcommand{\diag}{{\rm diag}}

\numberwithin{equation}{section}

\begin{document}

\title{Numerical validation of blow-up solutions with quasi-homogeneous compactifications}

\author{Kaname Matsue\thanks{Institute of Mathematics for Industry, Kyushu University, Fukuoka 819-0395, Japan {\tt kmatsue@imi.kyushu-u.ac.jp}} $^{,}$ \footnote{International Institute for Carbon-Neutral Energy Research (WPI-I$^2$CNER), Kyushu University, Fukuoka 819-0395, Japan}
\ and Akitoshi Takayasu\thanks{Faculty of Engineering, Information and Systems, University of Tsukuba, 1-1-1 Tennodai, Tsukuba, Ibaraki 305-8573, Japan ({\tt takitoshi@risk.tsukuba.ac.jp})}
}
\maketitle

\begin{abstract}
We provide a numerical validation method of blow-up solutions for finite dimensional vector fields admitting asymptotic quasi-homogeneity at infinity.
Our methodology is based on quasi-homogeneous compactifications containing a new compactification, which shall be called a quasi-parabolic compactification.
Divergent solutions including blow-up solutions then correspond to global trajectories of associated vector fields with appropriate time-variable transformation tending to equilibria on invariant manifolds representing infinity.
We combine standard methodology of rigorous numerical integration of differential equations with Lyapunov function validations around equilibria corresponding to divergent directions, which yields rigorous upper and lower bounds of blow-up times as well as rigorous profile enclosures of blow-up solutions.

\end{abstract}


\section{Introduction}

Our concern in this paper is blow-up solutions of the following initial value problem of an autonomous system of ordinary differential equations (ODEs) in $\mathbb{R}^n$:
\begin{equation}
\label{eqn:ODE}
\frac{dy(t)}{dt}=f (y(t)),\quad y(0)=y_0,
\end{equation}
where $t\in[0,T)$ with $0<T\le\infty$, $f:\mathbb{R}^n\to\mathbb{R}^n$ is a $C^1$ function and $y_0\in\mathbb{R}^n$.
We shall call a solution $\{y(t)\}$ of the initial value problem (\ref{eqn:ODE}) a {\em blow-up solution} if
\[
	t_{\max}:=\sup\left\{\bar t\mid \mbox{a solution $y\in C^1([0,\bar t))$ of \eqref{eqn:ODE} exists}\right\} < \infty.
\]
The maximal existence time $t_{\max}$ is then called the {\em blow-up time} of \eqref{eqn:ODE}.
Blow-up solutions can be seen in many dynaimcal systems generated by (partial) differential equations like nonlinear heat equations or Keller-Segel systems.
These are categorized as the presence of finite-time singularity in dynamical systems, and many researchers have broadly studied these phenomena from mathematical, physical, numerical viewpoints and so on.
Fundamental questions for blow-up problem are {\em whether or not a solution blows up} and, if does, {\em when, where, and how} it blows up.
In general blow-up phenomena depend on initial data.
Rigorous concrete detection of fundamental information of blow-up solutions as functions of initial data remains a nontrivial problem.
\par
\bigskip
Recently, authors and their collaborators have provided a {\em numerical validation procedure based on interval and affine arithmetics} for calculating rigorous blow-up profiles and their blow-up times \cite{TMSTMO}. 
The approach is based on {\em compactification} of phase space; embedding into a compact manifold $M$, possibly with boundary.
In this methodology, the infinity on the original phase space can correspond to a point on $\partial M$ or a specified point on $M$ called a point at infinity.
Combining a compactification with an appropriate time-scale transformation, called {\em time-variable desingularization}, suitable for given vector field, divergent solutions including blow-up solutions are characterized as global trajectories of the transformed vector field on $M$ tending to a point, such as an equilibrium $x_\ast$, on $\partial M$.
Finally, the {\em Lyapunov function validation} (\cite{MHY2016}) around $x_\ast\in \partial M$ is applied to derivation of a re-parameterization of trajectories so that we can validate rigorous lower and upper bounds of blow-up times $t_{\max}$ with numerical validations.
In the present methodology, (i) rigorous numerical integration of ODEs, (ii) eigenvalue validations, and (iii) polynomial estimates essentially realize numerical validations of blow-up solutions with their blow-up times.
\par
However, applicability of proposed methodology there is restricted to vector fields which are asymptotically {\em homogeneous} at infinity, since applied compactifications are assumed to respect {\em homogeneous} scalings.
In other words, verifications of blow-ups for differential equations possessing, say {\em quasi-homogeneous scaling laws} such as $h(u,v) := u^2-v$ may return meaningless information\footnote{
The function has a scaling law $h(ru, r^2 v) = r^2 h(u,v)$ holds for all $r\in \mathbb{R}$.
}.
If we apply such a numerical validation methodology to a broad class of differential equations, we have to choose appropriate compactifications which appropriately extracts information of dynamics at infinity.
\par
Inspired by the above work, the first author has discussed blow-up solutions for differential equations which are asymptotically {\em quasi-homogeneous} at infinity from the viewpoint of dynamical systems \cite{Mat}.
There a new quasi-homogeneous compactification called {\em quasi-Poincar\'{e} compactification} is defined as a quasi-homogeneous analogue of well-known Poincar\'{e} compactifications and as a global compactification alternative of well-known local compactifications which shall be called {\em directional compactifications} (e.g., \cite{DH1999} with a terminology {\em Poincar\'{e}-Lyapunov disks}).
By using the same essence as previous works about blow-up solutions \cite{EG2006, TMSTMO}, several blow-up solutions for asymptotically quasi-homogeneous vector fields can be characterized by trajectories on stable manifolds of \lq\lq hyperbolic invariant sets" on the boundary $\partial M$ of a compactified manifold $M$.
Moreover, such blow-up solutions completely characterize their blow-up rates from the growth rate of original vector fields.
The same characterizations also make sense for dynamical systems with directional compactifications.
A series of studies involving characterization of blow-up solutions in \cite{Mat} contain blow-up results in the previous work \cite{EG2006}, and the applications to numerical validations of blow-up solutions for systems of asymptotically quasi-homogeneous differential equations are expected.
\par
\bigskip
Our present aim is to provide numerical validation methodology of blow-up solutions for systems of  differential equations with asymptotic quasi-homogeneity at infinity.
It turns out that fundamental features of a {\em good} class of quasi-homogeneous compactifications  enable us to apply the same methodology as \cite{TMSTMO} to the present blow-up validations.
\par
The rest of this paper is organized as follows.
In Section \ref{section-QHC}, we define an admissible class of compactifications with given quasi-homogeneous type.
We see that our admissible class admits the same asymptotic properties at infinity as quasi-Poincar\'{e} compactifications introduced in \cite{Mat}.
As a nontrivial example, we also introduce a concrete compactification which is admissible in our sense, called a {\em quasi-parabolic compactification}.
This compactification is a quasi-homogeneous analogue of (homogeneous) parabolic compactifications \cite{EG2006, TMSTMO}.
Directional compactifications are also reviewed.
In Section \ref{section-dyn-infty}, we study vector fields and dynamics on compactified manifolds.
Under our admissible compactifications, we have a good correspondence of dynamical systems between on original phase spaces and on compactified manifolds. 
Moreover, as in the case of quasi-Poincar\'{e} compactifications, we can define desingularized vector fields on compactified manifolds so that {\em dynamics at infinity} makes sense.
Here we have a new essential result that, for $C^1$ vector field $f$ in the original problem, the desingularized vector field $g$ with quasi-parabolic compactifications becomes $C^1$ including the boundary of compactified manifolds corresponding to the infinity.
This property is very crucial because {\em the desingularized vector field $g$ with quasi-Poincar\'{e} compactifications is not always $C^1$ even if $f$ is sufficiently smooth}. Details are shown in \cite{Mat}.
The feature of quasi-parabolic compactifications enables us to study stability analysis for dynamical systems without any obstructions of regularity of vector fields.
In Section \ref{section-blow-up}, we provide criteria for validating blow-up solutions and numerical validation procedure for blow-up solutions with their blow-up times.
Our criteria consists of not only pure mathematical arguments but also numerical validation implementations for blow-up solutions.
Our arguments indicate that blow-up solutions correspond to {\em stable manifolds of asymptotically stable equilibria on $\partial M$}, which can be validated by standard techniques of dynamical systems with computer assistance.
We review a fundamental tool called {\em Lyapunov function}, which validates level surfaces around equilibria and is essential to estimate explicit enclosures of blow-up times.
We conclude Section \ref{section-blow-up} by providing concrete validation steps for blow-up solutions.
Finally, we demonstrate several numerical validation examples of blow-up solutions in Section \ref{section-numerical}.

%
%
\section{Quasi-homogeneous compactifications}
\label{section-QHC}

In this section, we introduce several compactifications of phase spaces which are appropriate for studying dynamics at infinity.
As an example of such appropriate ones, we define a {\em quasi-parabolic compactification}.
This compactification is an alternative of admissible, {\em homogeneous} ones discussed in e.g., \cite{EG2006}, and of quasi-Poincar\'{e} compactifications derived in \cite{Mat}.
Our present compactification is based on an appropriate scaling of vector-valued functions at infinity and quasi-homogeneous desingularization of singularities in dynamical systems (e.g., \cite{D1993}). Moreover, it overcomes the lack of smoothness of (transformed) vector fields at infinity mentioned later.
Firstly, we briefly review quasi-homogeneous vector fields.
Secondly, we introduce a class of compactifications called {\em (admissible) quasi-homogeneous compactifications}.
Thirdly, we define a quasi-parabolic compactification.
Finally, we review a well-known quasi-homogeneous (local) compactification which shall be called {\em directional} compactification.

%
%
\subsection{Quasi-homogeneous vector fields}
\label{section-QH}

First of all, we review a class of vector fields in our present discussions.
\begin{dfn}[Quasi-homogeneous vector fields, e.g., \cite{D1993}]\rm
Let $f : \mathbb{R}^n \to \mathbb{R}$ be a smooth function.
Let $\alpha_1,\cdots, \alpha_n, k \geq 1$ be natural numbers.
We say that $f$ is a {\em homogeneous function of type $(\alpha_1,\cdots, \alpha_n)$ and order $k$} if
\begin{equation*}
f(r^{\alpha_1}x_1,\cdots, r^{\alpha_n}x_n) = r^k f(x_1,\cdots, x_n),\quad \forall x\in \mathbb{R}^n,\quad r\in \mathbb{R}.
\end{equation*}
\par
Next, let $X = \sum_{j=1}^n f_j(x)\frac{\partial }{\partial x_j}$ be a smooth vector field on $\mathbb{R}^n$.
We say that $X$, or simply $f = (f_1,\cdots, f_n)$ is a {\em quasi-homogeneous vector field of type $(\alpha_1,\cdots, \alpha_n)$ and order $k+1$} if each component $f_j$ is a homogeneous function of type $(\alpha_1,\cdots, \alpha_n)$ and order $k + \alpha_j$.
\end{dfn}

For applications to general vector fields, in particular for dynamics near infinity, we define the following notion.
\begin{dfn}[Asymptotically quasi-homogeneous vector fields at infinity, \cite{Mat}]\rm
Let $f = (f_1,\cdots, f_n):\mathbb{R}^n \to \mathbb{R}^n$ be a smooth function.
We say that $X = \sum_{j=1}^n f_j(x)\frac{\partial }{\partial x_j}$, or simply $f$ is an {\em asymptotically quasi-homogeneous vector field of type $(\alpha_1,\cdots, \alpha_n)$ and order $k+1$ at infinity} if
\begin{equation*}
\lim_{r\to +\infty} r^{-(k+\alpha_j)}\left\{ f_j(r^{\alpha_1}x_1, \cdots, r^{\alpha_n}x_n) - r^{k+\alpha_j}(f_{\alpha,k})_j(x_1, \cdots, x_n) \right\}
 = 0
 \end{equation*}
holds uniformly for $(x_1,\cdots, x_n)\in S^{n-1}$ for some quasi-homogeneous vector field  $f_{\alpha,k} = ((f_{\alpha,k})_1,\cdots, (f_{\alpha,k})_n)$ of type $(\alpha_1,\cdots, \alpha_n)$ and order $k+1$.
\end{dfn}

The asymptotic quasi-homogeneity at infinity plays a key role in consideration of (polynomial) vector fields at infinity, which is shown later.

%
%
\subsection{Admissible quasi-homogeneous compactifications}
\label{section-definition}

Throughout successive sections, consider the (autonomous) polynomial vector field
\begin{equation}
\label{ODE-original}
y' = f(y),
\end{equation}
where $f : \mathbb{R}^n \to \mathbb{R}^n$ be a smooth function.
Throughout our discussions, we assume that $f$ is an asymptotically quasi-homogeneous vector field of type $\alpha = (\alpha_1,\cdots, \alpha_n)$ and order $k+1$ at infinity.

\begin{dfn}[Admissible quasi-homogeneous compactification]\rm
\label{dfn-qP}
Fix natural numbers $\alpha_1,\cdots, \alpha_n$.
Let $\beta_1,\cdots, \beta_n$ be natural numbers such that 
\begin{equation}
\label{LCM}
\alpha_1 \beta_1 = \alpha_2 \beta_2 = \cdots = \alpha_n \beta_n \equiv c \in \mathbb{N}.
\end{equation}
Define a functional $p(y)$ as
\begin{equation*}
p(y) := \left( y_1^{2\beta_1} +  y_2^{2\beta_2} + \cdots +  y_n^{2\beta_n} \right)^{1/2c}.
\end{equation*}
Define the mapping $T : \mathbb{R}^n \to \mathbb{R}^n$ as
\begin{equation*}
\quad T(y) = x,\quad x_i := \frac{y_i}{\kappa(y)^{\alpha_i}}.
\end{equation*}
We say that $T$ is an {\em (admissible) quasi-homogeneous compactification (of type $\alpha$)} if all the following conditions hold:
\begin{description}
\item[(A0)] $\kappa(y) > p(y)$ for all $y\in \mathbb{R}^n$,
\item[(A1)] $\kappa(y) \sim p(y)$ as $p(y)\to \infty$,
\item[(A2)] $\nabla \kappa(y) = ((\nabla \kappa(y))_1, \cdots, (\nabla \kappa(y))_n)$ satisfies
\begin{equation*}
(\nabla \kappa(y))_i \sim \frac{1}{\alpha_i}\frac{y_i^{2\beta_i-1}}{p(y)^{2c-1}}\quad \text{ as } p(y)\to \infty.
\end{equation*}
\item[(A3)] Letting $y_\alpha = (\alpha_1 y_1,\cdots, \alpha_n y_n)^T$ for $y\in \mathbb{R}^n$, we have $\langle y_\alpha, \nabla \kappa \rangle < \kappa(y)$ holds for any $y\in \mathbb{R}^n$.
\end{description}
\end{dfn}
The admissibility conditions (A0) $\sim$ (A3) come from fundamental properties of {\em quasi-Poincar\'{e} compactifications} $T = T_{qP}$ introduced in \cite{Mat}, which is defined by
$\kappa(y) = \left(1+p(y)^{2c}\right)^{1/2c}$.
By (\ref{LCM}), it immediately holds that $p(y)^{2c} = \kappa(y)^{2c}p(x)^{2c}$.
The condition (A1) indicates that $p(x)\to 1$ as $p(y)\to \infty$, and vice versa.
In particular, by the condition (A0), $T$ maps $\mathbb{R}^n$ into
\begin{equation*}
\mathcal{D} := \{x\in \mathbb{R}^n \mid p(x) < 1\}.
\end{equation*}
The infinity in the original coordinate then corresponds to a point on
\begin{equation*}
\mathcal{E} = \{x \in \mathbb{R}^n \mid p(x) = 1\}.
\end{equation*}
\begin{dfn}[cf. \cite{Mat}]\rm
We call the boundary $\mathcal{E}$ the {\em horizon}.
\end{dfn}

\begin{rem}\rm
The simplest choice of the natural number $c$ is the least common multiple of $\alpha_1,\cdots, \alpha_n$.
Once we choose such $c$, we can determine the $n$-tuples of natural numbers $\beta_1,\cdots, \beta_n$ uniquely.
The choice of natural numbers in (\ref{LCM}) is essential to desingularize vector fields at infinity, as shown below.
\end{rem}

The horizon determines directions where solution trajectories diverge.
\begin{dfn}\rm
We say that a solution orbit $y(t)$ of (\ref{ODE-original}) with the maximal existence time $(a,b)$, possibly $a = -\infty$ and $b = +\infty$, {\em tends to infinity in the direction $x_\ast \in \mathcal{E}$ associated with the quasi-Poncar\'{e} functional $p$} (as $t\to a+0$ or $b-0$) if
\begin{equation*}
p(y(t))\to \infty,\quad \left(\frac{y_1}{\kappa(y)^{\alpha_1}}, \cdots, \frac{y_n}{\kappa(y)^{\alpha_n}}\right)\to x_\ast\quad \text{ as }t\to a+0 \text{ or }b-0.
\end{equation*}
\end{dfn}

Now compute the Jacobian matrix of $T$ for verifying its bijectivity.
Direct computations yield
\begin{equation*}
\frac{\partial x_i}{\partial y_j} = \kappa^{-\alpha_i}\left( \delta_{ij} - \kappa^{-1} \alpha_i y_i \frac{\partial \kappa}{\partial y_j}\right) 
\end{equation*}
with the matrix form
\begin{align*}
&J = \left(\frac{\partial x_i}{\partial y_j}\right)_{i,j = 1,\cdots, n} = A_{\alpha} \left( I_n - \kappa^{-1} y_\alpha (\nabla \kappa)^T\right),\\
&A_\alpha = \diag(\kappa^{-\alpha_1}, \cdots, \kappa^{-\alpha_n}),\quad y_\alpha = (\alpha_1 y_1,\cdots, \alpha_n y_n)^T.
\end{align*}

We following arguments in \cite{EG2006}, for any (column) vectors $y,z\in \mathbb{R}^n$, to have 
\begin{align*}
(I_n + \beta yz^T)(I_n + \beta yz^T) &= I + (\beta + \delta)yz^T + \beta \delta yz^T yz^T\\
	&= I + (\beta + \delta + \beta \delta \langle z,y\rangle)yz^T,
\end{align*}
so $I+\delta yz^T = (I+\delta yz^T )^{-1}$ if $\delta = -\beta / (1 + \beta \langle z,y\rangle)$.
\par
In this case, we choose $\beta = -\kappa^{-1}, y = y_\alpha, z = \nabla \kappa$ and have
\begin{equation*}
\left(\frac{\partial y_j}{\partial x_i}\right) = \left(\frac{\partial x_i}{\partial y_j}\right)^{-1} =  \left( I_n - \frac{1}{\kappa - \langle y_\alpha, \nabla \kappa \rangle } y_\alpha (\nabla \kappa)^T\right)A_{\alpha}^{-1}
\end{equation*}
Now we have
\begin{equation*}
\frac{\partial \kappa}{\partial y_j} = \frac{\partial }{\partial y_j} \left(1+\sum_{i=1}^n y_i^{2\beta_i} \right)^{\frac{1}{2c}} = \frac{\beta_j}{c} \left(1+\sum_{i=1}^n y_i^{2\beta_i} \right)^{\frac{1}{2c}-1} y_j^{2\beta_j-1}
= \frac{\beta_j}{c\kappa^{2c-1}} y_j^{2\beta_j-1}
\end{equation*}
and hence
\begin{align*}
\kappa^{2c-1}\left(\kappa - \langle y_\alpha, \nabla \kappa \rangle \right) &= \kappa^{2c-1}\left(\kappa - \sum_{j=1}^n \alpha_j y_j \frac{\beta_j}{c\kappa^{2c-1}} y_j^{2\beta_j-1} \right) = \left\{ (1+p(y)^{2c}) - p(y)^{2c}\right\} > 0,
\end{align*}
which indicates that the transformation $T$ as well as $T^{-1}$ are $C^1$ locally bijective including $y=0$.
On the other hand, the map $T$ maps any one-dimensional curve $y = (r^{\alpha_1}v_1, \cdots, r^{\alpha_n}v_n)$, $0\leq r < \infty$, with some fixed direction $v\in \mathbb{R}^n$, into itself (cf.  \cite{Mat}).
For continuous mappings from $\mathbb{R}$ to $\mathbb{R}$, local bijectivity implies global bijectivity.
Consequently, (A3) guarantees also the global bijectivity of $T$.
Summarizing these arguments, we obtain the following proposition.
\begin{prop}
\label{prop-adm}
The functional $\kappa$ defining the admissible quasi-homogeneous compactificaton $T$ is a bijection from $\mathbb{R}^n$ onto $\mathcal{D} = \{x\in \mathbb{R}^n \mid p(x) < 1\}$.
\end{prop}



Note that the above argument is completely parallel to arguments of bijectivity of the quasi-Poincar\'{e} compactification \cite{Mat}.
\par
Four properties (A0) $\sim $ (A3) in Definition \ref{dfn-qP} will play central roles in the theory of, which shall be called, {\em quasi-homogeneous compactifications} and associated dynamics. 
Indeed, in the case of {\em homogeneous} compactifications, namely $\alpha_1 = \cdots = \alpha_n = \beta_1 = \cdots = \beta_n = 1$, these conditions describe {\em admissibility} of compactifications \cite{EG2006}, which play central roles to dynamics at infinity.
The Poincar\'{e} compactification; namely the quasi-Poincar\'{e} compactification of type $(1,\cdots, 1)$, is the prototype of other admissible homogeneous compacifications such as parabolic ones (e.g., \cite{EG2006, TMSTMO}), and hence properties (A0) $\sim $ (A3) which quasi-Poincar\'{e} compactifications possess will be appropriate to define an \lq\lq admissible\rq\rq class of quasi-homogeneous compactifications.

%
%
\subsection{Quasi-parabolic compactification}

Here we introduce an example of quasi-homogeneous compactifications other than quasi-Poincar\'{e} ones, which is an analogue of {\em parabolic compactifications} discussed in \cite{EG2006, TMSTMO}.
\par
Let the type $\alpha = (\alpha_1,\cdots, \alpha_n)\in \mathbb{Z}_{>0}^n$ fixed.
Let $\{\beta_i\}_{i=1}^n$ and $c$ be a collection of natural numbers satisfying (\ref{LCM}).
For any $x\in \mathcal{D}$, define $y\in \mathbb{R}^n$ by 
\begin{equation*}
S(x) = y,\quad y_j = \frac{x_j}{(1- p(x)^{2c})^{\alpha_j}},\quad j=1,\cdots, n.
\end{equation*}
Let $\tilde \kappa_\alpha(x) := (1-p(x)^{2c})^{-1}$, which satisfies $\tilde \kappa_\alpha(x) \geq 1$ for all $x\in \mathcal{D}$.
Moreover, $y\not = 0$ implies $\tilde \kappa_\alpha(x) > 1$.
We also have 
\begin{equation*}
p(y)^{2c} = \tilde \kappa_\alpha(x)^{2c}p(x)^{2c} = \tilde \kappa_\alpha(x)^{2c}\left(1-\frac{1}{\tilde \kappa_\alpha(x)}\right).
\end{equation*}
This equality indicates that $p(y) = p(S(x)) < \tilde \kappa_\alpha(x)$ holds for all $x\in \mathcal{D}$.

\begin{lem}
\label{lem-zero-compactification}
Let $F_y(\kappa) := \kappa^{2c} - \kappa^{2c-1} - p(y)^{2c}$.
Then, for any fixed $y\in \mathbb{R}^n\setminus \{0\}$, $F_y$ has the unique zero in $\{\kappa > \max \{1,p(y)\}\}$.
\end{lem}
\begin{proof}
Observe that $F_y(1) = -p(y)^{2c} < 0$ and $F_y(p(y)) = -p(y)^{2c-1} < 0$.
Moreover,
\begin{equation*}
\frac{dF_y}{d\kappa}(\kappa) = 2c\kappa^{2c-1} - (2c-1)\kappa^{2c-2} > 2c\kappa^{2c-2}(\kappa - 1) \geq 0,
\end{equation*}
which shows that $F_y$ is strictly increasing in $\{\kappa \geq 1\}$.
Here consider the following two cases.
\begin{description}
\item[Case 1 : $0 < p(y) <1$.] 
\end{description}
Letting $\kappa = c_y$ with a constant $c_y\geq 1$, we have
\begin{equation*}
F_y(\kappa) = c_y^{2c} - c_y^{(2c-1)} -p(y)^{2c} > c_y^{2c} - c_y^{(2c-1)} -1,
\end{equation*}
which can be positive by choosing $c_y$ sufficiently large.

\begin{description}
\item[Case 2 : $p(y) \geq 1$.] 
\end{description}
Letting $\kappa = c_y p(y)^{2c/(2c-1)}$ with a constant $c_y$, we have
\begin{equation*}
F_y(\kappa) = c_y^{2c} p(y)^{4c^2/(2c-1)} - (c_y^{(2c-1)}+1)p(y)^{2c} \geq \{c_y^{2c} - (c_y^{(2c-1)}+1) \} p(y)^{2c}
\end{equation*}
Choosing $c_y$ sufficiently large, we obtain $c_y^{2c} - (c_y^{(2c-1)}+1) >0$, which implies $F_y(\kappa) > 0$.
\par
In both cases, the intermediate theorem can be applied to the existence of unique zero of $F_y$.
\end{proof}
The above lemma determines the unique zero $\kappa(y)$ such that $F_y(\kappa(y)) = 0$.
In particular, $\tilde \kappa_\alpha(x)$ satisfies $F_y(\tilde \kappa_\alpha(x)) = 0$.
By the uniqueness of zero and the definition of $S$, for any $y\in \mathbb{R}^n\setminus \{0\}$, define
\begin{equation*}
\kappa(y) \equiv \kappa(S(x)) := \tilde \kappa_\alpha(x).
\end{equation*}
The implicit function theorem for $F_y(\kappa)=0$ shows that $\kappa(y)$ is a smooth function of $p$, and therefore of $y$ including $y=0$ as $\kappa(0)\equiv 1$.
We are then ready to the new compactification mapping $y$ to $x$.

\begin{dfn}[Quasi-parabolic compactification]\rm
Let the type $\alpha = (\alpha_1,\cdots, \alpha_n)\in \mathbb{Z}_{>0}^n$ fixed.
Let $\{\beta_i\}_{i=1}^n$ and $c$ be a collection of natural numbers satisfying (\ref{LCM}).
Define $T_{para}:\mathbb{R}^n\to \mathcal{D}$ as
\begin{equation*}
T_{para}(y) := x,\quad x_i = \frac{y_i}{\kappa(y)^{\alpha_i}},
\end{equation*}
where $\kappa(y) = \tilde \kappa_\alpha(x)$ is the unique zero of $F_y(\kappa)=0$ given in Lemma \ref{lem-zero-compactification}.
We say $T_{para}$ the {\em quasi-parabolic compactification (with type $\alpha$)}.
\end{dfn}

\begin{thm}
Let the type $\alpha = (\alpha_1,\cdots, \alpha_n)\in \mathbb{Z}_{>0}^n$ fixed.
Let $\{\beta_i\}_{i=1}^n$ and $c$ be a collection of natural numbers satisfying (\ref{LCM}).
Then the quasi-parabolic compactification $T_{para}$ is an admissible quasi-homogeneous compactification.
In particular, $T_{para}^{-1}=S$.
\end{thm}

\begin{proof}
For $\kappa \geq \max \{1,p(y)\}$, we have
\begin{equation*}
\kappa(y)^{2c} - p(y)^{2c} = \kappa(y)^{2c-1} > 0,
\end{equation*}
which is (A0).
\par
From the identity $p(y)^{2c} = \kappa(y)^{2c}p(x)^{2c}$, and $p(y) \to \infty$ as $p(x)\to 1$ and vice versa by the definition of $S$, we have $p(y) / \kappa(y) \to 1$ as $p(y)\to \infty$, which shows (A1).
\par
Differentiating the identity $F_y(\kappa) = \kappa^{2c} - \kappa^{2c-1} - p^{2c}\equiv 0$ with respect to $p$, we obtain
\begin{equation*}
2c\kappa^{2c-1}\frac{d\kappa}{dp} - (2c-1)\kappa^{2c-2}\frac{d\kappa}{dp} =  2c p^{2c-1},
\end{equation*}
namely,
\begin{equation*}
\frac{d\kappa}{dp} = \frac{2c p^{2c-1}}{2c\kappa^{2c-1} - (2c-1)\kappa^{2c-2}}.
\end{equation*}
The denominator of the right-hand side is positive for $\kappa \geq 1$.
Since $p = p(y)$ is smooth, then $\kappa = \kappa(y)$ can be regarded as the composition of smooth functions $p=p(y)$ and $\kappa = \kappa(p)$.
In particular, $\kappa=\kappa(y) \equiv \kappa(p(y))$ is $C^1$ with respect to $y$ and thus
\begin{align*}
(\nabla_y \kappa(y))_j &= \frac{d\kappa}{dp}\frac{\partial p}{\partial y_j} = \frac{2c p^{2c-1}}{2c\kappa^{2c-1} - (2c-1)\kappa^{2c-2}}\cdot \frac{1}{2c} p(y)^{1-2c} \cdot 2\beta_j y_j^{2\beta_j-1}\\
	&= \frac{2\beta_j y_j^{2\beta_j-1}}{2c\kappa^{2c-1} - (2c-1)\kappa^{2c-2}} = \frac{2\beta_j y_j^{2\beta_j-1}}{2c\kappa^{2c-1} \left(1- \frac{2c-1}{2c}\kappa^{-1}\right)}.
\end{align*}
By (A1), we have
\begin{equation*}
(\nabla_y \kappa(y))_j \sim \frac{2\beta_j y_j^{2\beta_j-1}}{2cp(y)^{2c-1}} = \frac{y_j^{2\beta_j-1}}{\alpha_j p(y)^{2c-1}}\quad \text{ as }\quad p(y)\to \infty,
\end{equation*}
which shows (A2).
\par
Next, check (A3).
We have
\begin{equation*}
\langle y_\alpha, \nabla_y \kappa\rangle = \sum_{j=1}^n \alpha_j y_j \frac{2\beta_j y_j^{2\beta_j-1}}{2c\kappa^{2c-1} \left(1- \frac{2c-1}{2c}\kappa^{-1}\right)} = \frac{2c p(y)^{2c}}{2c\kappa^{2c-1} \left(1- \frac{2c-1}{2c}\kappa^{-1}\right)}
\end{equation*}
and it is sufficient to show $\kappa \left\{2c\kappa^{2c-1} \left(1- \frac{2c-1}{2c}\kappa^{-1}\right)\right\} > 2c p(y)^{2c}$ for our statement.
Let
\begin{equation*}
G(y) := \kappa \left\{2c\kappa^{2c-1} \left(1- \frac{2c-1}{2c}\kappa^{-1}\right)\right\} - 2c p(y)^{2c}.
\end{equation*}
Then
\begin{equation*}
G(y) = 2c\kappa^{2c} - (2c-1)\kappa^{2c-1} - 2cp(y)^{2c} > 2c(\kappa^{2c} - \kappa^{2c-1} - p(y)^{2c}) = 0
\end{equation*}
and we obtain (A3).
\par
As a consequence, $T_{para}$ is an admissible quasi-homogeneous compactification.
In particular, $T_{para}:\mathbb{R}^n\to \mathcal{D}$ is a surjective $C^1$-diffeomorphism by Proposition \ref{prop-adm}.
Observe that 
\begin{align*}
S\circ T_{para}(y) &= S\left(\frac{y_1}{\kappa(y)^{\alpha_1}}, \cdots, \frac{y_n}{\kappa(y)^{\alpha_n}}\right)\\
	&= \left(\frac{y_1}{\kappa(y)^{\alpha_1}(1-\bar p^{2c})^{\alpha_1}}, \cdots, \frac{y_n}{\kappa(y)^{\alpha_n}(1-\bar p^{2c})^{\alpha_n}}\right)\\
	&=(y_1,\cdots, y_n) \equiv y,
\end{align*}
where 
\begin{equation}
\label{identity-kappa-p}
\bar p^{2c} = \frac{p(y)^{2c}}{\kappa(y)^{2c}} = p(x)^{2c}\quad \text{ and }\quad (1-\bar p^{2c})^{-1} = \tilde \kappa_\alpha(x) \equiv \kappa(y).
\end{equation}
Similarly,
\begin{align*}
T_{para}\circ S(x) &= T_{para}\left(\frac{x_1}{(1-p(x)^{2c})^{\alpha_1}}, \cdots, \frac{x_n}{(1-p(x)^{2c})^{\alpha_n}}\right)\\
	&= \left(\frac{x_1}{(1-p(x)^{2c})^{\alpha_1}\kappa(y)^{\alpha_1}}, \cdots, \frac{x_n}{(1-p(x)^{2c})^{\alpha_n}\kappa(y)^{\alpha_n}}\right)\\
	&=(x_1,\cdots, x_n)\equiv x,
\end{align*}
which follows from the identity (\ref{identity-kappa-p}).
Consequently, $S=T_{para}^{-1}$ holds and the proof is completed.
\end{proof}

\begin{rem}
The name {\em quasi-\lq\lq parabolic\rq\rq} of $T_{para}$ comes from the homogeneous parabolic-type compactification; namely, $T_{para}$ with $(\alpha_1,\cdots, \alpha_n) = (1,\cdots, 1)$ and $c=1$.
In the homogeneous case, $T_{para}$ is the composite of the mapping from $\mathbb{R}^n$ to a parabolic hypersurface $\{x_1^2 + \cdots + x_n^2 = x_{n+1}\}\subset \mathbb{R}^{n+1}$ and the projection $(x_1,\cdots, x_n, x_{n+1})\mapsto (x_1,\cdots, x_n)$.
In the homogeneous case $\alpha = (1,\cdots, 1)$ and $c=1$, $\kappa = \kappa(y)$ is explicitly given as
$\kappa(y) = \frac{1}{2}\left(1 + \sqrt{1+4\sum_{i=1}^n y_i^2}\right)$, which is also calculated from $F_y(\kappa) = 0$.
See \cite{EG2006, TMSTMO} for details.
Illustrations of parabolic and quasi-parabolic compactifications in two-dimensional situations are shown in Figure \ref{fig-quasi-poincare}.
\end{rem}

The quasi-parabolic compactification is an nontrivial example of admissible quasi-homogeneous compactifications.
The biggest difference from quasi-Poincar\'{e} compactification is that the functional $\tilde \kappa_\alpha(x)$ does not contain any radicals.
This property unconditionally guarantees {\em the $C^1$ smoothness of the desingularized vector field  of good $f$ on $\overline{\mathcal{D}}$}.
In particular, {\em the stability analysis at infinity} is available.
Details are discussed in Section \ref{section-dyn-infty}.

\begin{figure}[htbp]\em
\begin{minipage}{0.5\hsize}
\centering
\includegraphics[width=7.0cm]{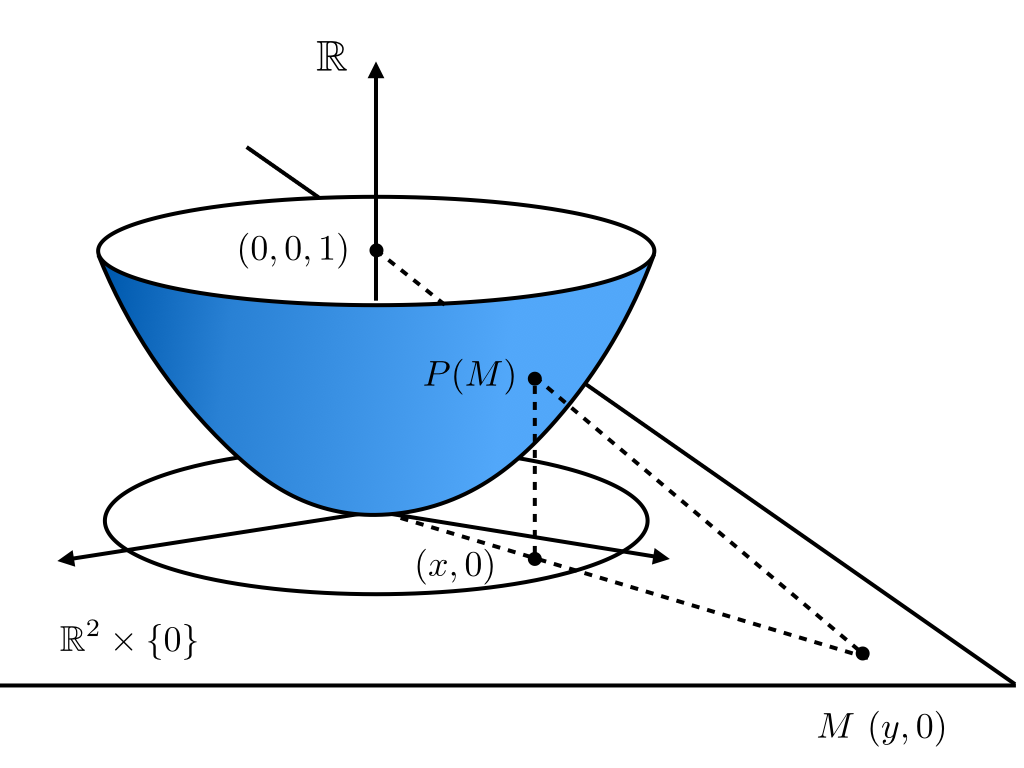}
(a)
\end{minipage}
\begin{minipage}{0.5\hsize}
\centering
\includegraphics[width=7.0cm]{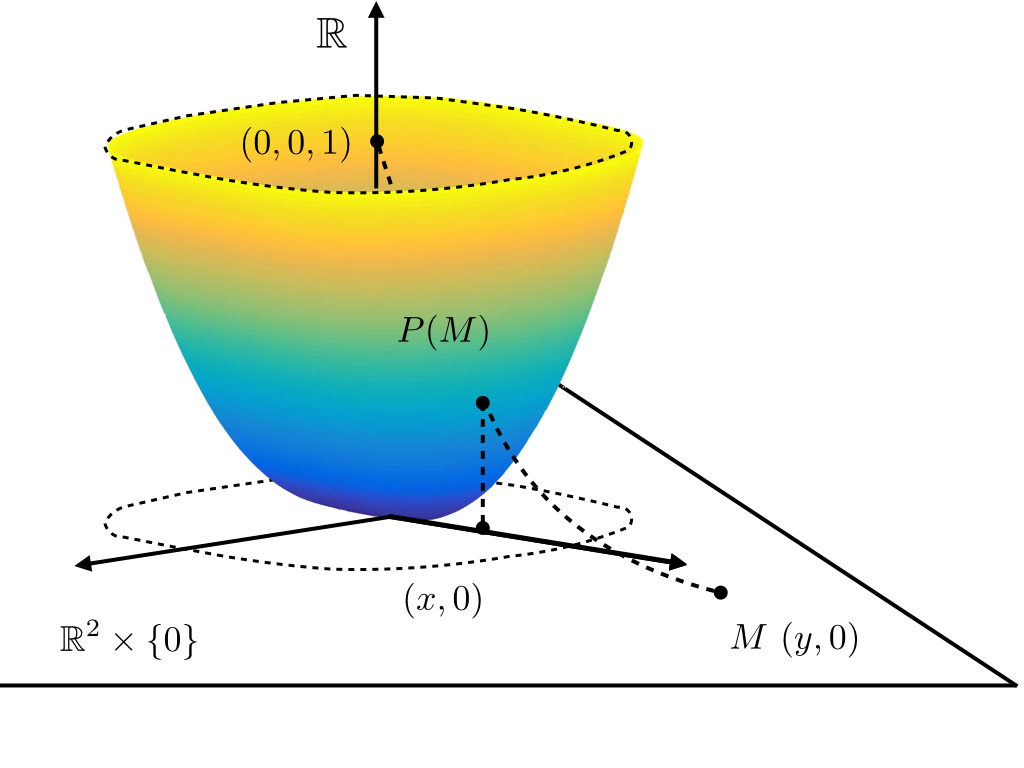}
(b)
\end{minipage}
\caption{Parabolic and quasi-parabolic compactifications with type $(2,1)$ for $\mathbb{R}^2$}
\label{fig-quasi-poincare}
Surfaces drawn here are (a) : $\mathcal{H} = \{(y_1,y_2,\zeta)\mid  y_1^2 + y_2^2 = \zeta\}$ ({\em parabolic compactification}), and (b) : $\mathcal{H}_\alpha = \{(y_1,y_2,\zeta)\mid y_1^2 + y_2^4 = \zeta\}$ ({\em quasi-parabolic compactification with type $(2,1)$}).
\par
In both figures, the original phase space corresponds to $\mathbb{R}^2\times \{0\}\subset \mathbb{R}^{2+1}$ in the extended space.
In the case of (a), the type $\alpha$ is chosen to be $(1,1)$.
The point $P(M)$ show the intersection point between $(0, 0,1)$ and the given point $M\in \mathbb{R}^2$ on $\mathcal{H}$ and $\mathcal{H}_\alpha$ respectively, through the curve $C_\alpha(y) = \{((1-\zeta)^{\alpha_1} y_1, (1-\zeta)^{\alpha_2} y_2, \zeta)\}$.
Note that the curve $C_\alpha$ is just a straight line in the case of homogeneous compactification $\alpha = (1,1)$.
The projections of $P(M)$ onto the original phase space; $(x,0)$, are the images of (quasi-)parabolic compactifications, respectively.
These observations can be easily generalized to $\mathbb{R}^n$.
\end{figure}

%
%
\subsection{Directional compactifications}
\label{section-dir}

There are several other compactifications reflecting (asymptotic) quasi-homogeneity of vector fields at infinity.
For example, the transform $y = (y_1,\cdots, y_n) \mapsto (s,x) \equiv (s, x_1,\cdots, x_{i-1}, x_{i+1}, \cdots, x_n)$ given by
\begin{equation}
\label{dir-cpt}
y_j = \frac{x_j}{s^{\alpha_j}}\quad (j\not = i),\quad y_i = \pm \frac{1}{s^{\alpha_i}}
\end{equation}
is a kind of compactifications, which corresponds the infinity to the subspace $\{s=0\}\equiv \mathcal{E}$.
We shall call such a compactification a {\em directional compactification with the type $\alpha = (\alpha_1,\cdots, \alpha_n)$}, according to \cite{Mat}.
The set $\mathcal{E} = \{s=0\}$ is called {\em the horizon}.
This compactification is geometrically characterized as a local coordinate of quasi-Poincar\'{e} hemisphere of type $\alpha$:
\begin{equation*}
\mathcal{H}_\alpha := \left\{(y_1,\cdots, y_n, s)\in \mathbb{R}^{n+1}\mid \frac{1}{(1+p(y)^{2c})} \sum_{i=1}^n y_i^{2\beta_i} + s^{2c} = 1\right\},
\end{equation*}
at $(x_1, \cdots, x_n, s) = (0,\cdots, 0, x_i = \pm 1, 0, \cdots, 0, 0)$. See \cite{Mat} for details.
Note that, unlike admissible quasi-homogeneous compactifications in Definition \ref{dfn-qP}, the coordinate representation (\ref{dir-cpt}) only makes sense in $\{\pm y_i > 0\}$, in which sense directional compactifications are {\em local} ones.
In particular, whenever we consider trajectories whose $y_i$-component can change the sign, we have to take care of transformations among coordinate neighborhoods, which is quite tough for numerical integration of differential equations.
Nevertheless, this compactification is still a very powerful tool {\em if we consider solutions near infinity whose $y_i$-component is known a priori to have identical sign}.

%
%
\section{Compactifications and dynamics at infinity}
\label{section-dyn-infty}

In this section, we calculate the vector field (\ref{ODE-original}) under the admissible quasi-homogeneous compactification $T$.
Regard $\kappa$ in the definition of $T$ as a function of $y$.
Integers $\{\beta_i\}_{i=1}^n$ and $c$ in the definition of $T$ are assumed to satisfy (\ref{LCM}).
Differentiating $x = T(y)$ with respect to $t$, we have 

\begin{align*}
x_i' &= \left(\frac{y_i}{\kappa^{\alpha_i}}\right)' = \frac{y_i'}{\kappa^{\alpha_i}} -  \frac{\alpha_i y_i \kappa^{\alpha_i-1}}{\kappa^{2\alpha_i}}\kappa' \\
	&= \frac{y_i'}{\kappa^{\alpha_i}} -  \frac{\alpha_i y_i }{\kappa^{\alpha_i+1}} \langle \nabla \kappa, y'\rangle\\
	&= \frac{f_i(y)}{\kappa^{\alpha_i}} -  \frac{\alpha_i y_i }{\kappa^{\alpha_i+1}} \langle \nabla \kappa, f(y)\rangle.
\end{align*}
Namely,
\begin{equation}
\label{ODE-poincare}
x' = A_\alpha \left(f_i(y) - \kappa^{-1}\langle f, \nabla \kappa \rangle y_\alpha \right)
\end{equation}

We have the one-to-one correspondence of {\em bounded} equilibria, which helps us with detecting dynamics at infinity.
\begin{prop}
The quasi-homogeneous compactification $T$ maps bounded equilibria of (\ref{ODE-original}) in $\mathbb{R}^n$ into equilibria of (\ref{ODE-poincare}) in $\mathcal{D}$, and vice versa. 
\end{prop}
\begin{proof}
Suppose that $y_\ast$ is an equilibrium of (\ref{ODE-original}), i.e., $f(y_\ast) = 0$.
Then the right-hand side of (\ref{ODE-poincare}) obviously vanishes at the corresponding $x_\ast$.
\par
Conversely, suppose that the right-hand side of (\ref{ODE-poincare}) vanishes at a point $x\in \mathcal{D}, p(x) < 1$: namely,
\begin{equation*}
f(\kappa x) - \kappa(y)^{-1}\langle \nabla \kappa, f(\kappa x) \rangle y_\alpha = 0.
\end{equation*}
Multiplying $\nabla \kappa$, we have
\begin{equation*}
\langle \nabla \kappa, f(\kappa x) \rangle \left(1- \kappa(y)^{-1}\langle \nabla \kappa, y_\alpha \rangle \right) = 0.
\end{equation*}
Due to (A3), we have $|\kappa(y)^{-1}\langle \nabla \kappa, y_\alpha \rangle| < 1$ and hence $\langle \nabla \kappa, f(\kappa x) \rangle = 0$.
Thus we have $f(y) = f(\kappa x) = 0$ by the assumption.
\end{proof}

Next we discuss the dynamics at infinity. Denoting
\begin{equation}
\label{f-tilde}
\tilde f_j(x_1,\cdots, x_n) := \kappa^{-(k+\alpha_j)} f_j(\kappa^{\alpha_1}x_1, \cdots, \kappa^{\alpha_n}x_n),\quad j=1,\cdots, n,
\end{equation}
we have
\begin{align}
\label{vectorfield-cw}
\notag
x_i' &= \frac{\kappa^{k+\alpha_i} \tilde f_i(x)}{\kappa^{\alpha_i}} -  \frac{\alpha_i \kappa^{\alpha_i}x_i }{\kappa^{\alpha_i+1}} \sum_{j=1}^n (\nabla \kappa)_j \kappa^{k+\alpha_j}\tilde f_j(x)\\
	&= \kappa^k \tilde f_i(x) - \alpha_i x_i \sum_{j=1}^n (\nabla \kappa)_j \kappa^{k+\alpha_j - 1}\tilde f_j(x).
\end{align}
Since $\kappa \to \infty$ as $p(x) \to 1$, then the vector field has singularities at infinity, while $\tilde f_j(x)$ themselves are continuous on $\overline{\mathcal{D}}$ because of the asymptotic quasi-homogeneity of $f$.
Nevertheless, admissibility of compactifications yields the following observation.

\begin{lem}
\label{lem-order}
The right-hand side of (\ref{vectorfield-cw}) is $O(\kappa^k)$ as $\kappa\to \infty$.
In other words, the order with respect to $\kappa$ is independent of $i$.
\end{lem}

\begin{proof}
By admissibility (A1)-(A2), we have
\begin{equation*}
(\nabla \kappa(y))_i \sim \frac{1}{\alpha_i}\frac{y_i^{2\beta_i-1}}{\kappa(y)^{2c-1}} = \frac{1}{\alpha_i}\frac{\kappa^{\alpha_i(2\beta_i-1)}x_i^{2\beta_i-1}}{\kappa^{2c-1}} = \frac{1}{\alpha_i}\frac{x_i^{2\beta_i-1}}{\kappa^{\alpha_i-1}} \quad \text{ as }\quad p(y)\to \infty,
\end{equation*}
where we used the condition $\alpha_j \beta_j \equiv c$ for all $j$ from (\ref{LCM}).
Therefore the vector field (\ref{vectorfield-cw}) near infinity becomes
\begin{align}
\notag
x_i' &\sim \kappa^k \tilde f_i(x) - \alpha_i x_i \sum_{j=1}^n \frac{1}{\alpha_j}\frac{x_j^{2\beta_j-1}}{\kappa^{\alpha_j-1}} \kappa^{k+\alpha_j - 1}\tilde f_j(x)\\
\label{vectorfield-cw-2}
	&= \kappa^k \left\{ \tilde f_i(x) - \alpha_i x_i \sum_{j=1}^n \frac{x_j^{2\beta_j-1}}{\alpha_j} \tilde f_j(x)\right\} \quad \text{ as }\quad \kappa \to \infty.
\end{align}
Since $\tilde f_i$ is $O(1)$ as $\kappa\to \infty$, then right-hand side of (\ref{vectorfield-cw-2}) is $O(\kappa^k)$ as $\kappa\to \infty$.
\end{proof}
Lemma \ref{lem-order} leads to introduce the following transformation of time variable.

\begin{dfn}[Time-variable desingularization]\rm
Define the new time variable $\tau$ depending on $y$ by
\begin{equation}
\label{time-desing}
d\tau = \kappa(y(t))^{k} dt
\end{equation}
namely,
\begin{equation*}
t - t_0 = \int_{\tau_0}^\tau \frac{d\tau}{\kappa(y(\tau))^k},
\end{equation*}
where $\tau_0$ and $t_0$ denote the correspondence of initial times, and $y(\tau)$ is the solution trajectory $y(t)$ under the new time variable $\tau$.
We shall call (\ref{time-desing}) {\em the time-variable desingularization of (\ref{vectorfield-cw}) of order $k+1$}.
\end{dfn}

\begin{equation}
\label{ODE-desing}
\dot x_i \equiv \frac{dx_i}{d\tau} = \tilde f_i(x) - \alpha_i x_i \sum_{j=1}^n (\nabla \kappa)_j \kappa^{\alpha_j - 1}\tilde f_j(x) \equiv g_i(x).
\end{equation}

Summarizing the above observation, we have the extension of dynamics at infinity.
\begin{prop}[Extension of dynamics at infinity]
\label{prop-ext}
Let $\tau$ be the new time variable given by (\ref{time-desing}).
Then the dynamics (\ref{ODE-original}) can be extended to the infinity in the sense that the vector field $g$ is continuous on $\overline{\mathcal{D}}$.
\end{prop}

\begin{proof}
The component-wise desingularized vector field (\ref{ODE-desing}) is obviously continuous on $\overline{\mathcal{D}}$ since this consists of product and sum of continuous functions $x_i$'s and $\tilde f_i$'s on $\overline{\mathcal{D}}$.
\end{proof}

\begin{ex}[Extension of vector fields via quasi-parabolic compactifications]
In the case of quasi-parabolic compactification, $\nabla \kappa$ is given by
\begin{equation*}
(\nabla_y \kappa(y))_j  = \frac{2\beta_j y_j^{2\beta_j-1}}{2c\kappa^{2c-1} \left(1- \frac{2c-1}{2c}\kappa^{-1}\right)} = \frac{\kappa^{2c - \alpha_j} x_j^{2\beta_j-1}}{\alpha_j \kappa^{2c-1} \left(1- \frac{2c-1}{2c}\kappa^{-1}\right)} = \frac{x_j^{2\beta_j-1}}{\alpha_j \kappa^{\alpha_j-1} \left(1- \frac{2c-1}{2c}\kappa^{-1}\right)}.
\end{equation*}
We can see that $g$ in (\ref{ODE-desing}) can be extended to be $C^0$ on $\overline{\mathcal{D}}$.
\end{ex}

Proposition \ref{prop-ext} shows that the \lq\lq dynamics and invariant sets at infinity" make sense.
For example, \lq\lq {\em equilibria at infinity}" defined below are well-defined.

\begin{dfn}[Equilibria at infinity]\rm
We say that the vector field (\ref{ODE-original}) has an {\em equilibrium at infinity} in the direction $x_\ast$ if $x_\ast$ is an equilibrium of (\ref{ODE-desing}) on $\partial {\mathcal{D}}$.
\end{dfn}

Now divergent solutions are described in terms of trajectories asymptotic to equilibria on the horizon for desingularized vector fields.
\begin{thm}[Divergent solutions and asymptotic behavior]
\label{thm-diverge}
Let $y(t)$ be a solution of (\ref{ODE-original}) with the interval of maximal existence time $(a,b)$, possibly $a=-\infty$ and $b=+\infty$.
Assume that $y$ tends to infinity in the direction $x_\ast$ as $t\to b-0$ or $t\to a+0$.
Then $x_\ast$ is an equilibrium of (\ref{ODE-desing}) on $\mathcal{E}$.
\end{thm}

\begin{proof}
The property $b = \sup \{t\mid y(t) \text{ is a solution of (\ref{ODE-original})}\}$ corresponds to the property that
\begin{equation*}
\sup \{\tau\mid x(\tau) = T(y(t)) \text{ is a solution of (\ref{ODE-desing}) in the time variable $\tau$}\} = \infty.
\end{equation*}
Indeed, if not, then $\tau \to \tau_0 < \infty$ and $\lim_{\tau \to \tau_0-0}x(\tau) = x_\ast$ as $t\to b-0$.
The condition $x(\tau) = x_\ast$ is the regular initial condition of (\ref{ODE-desing}).
The vector field (\ref{ODE-desing}) with the new initial point $x(\tau) = x_\ast$ thus has a locally unique solution $x(\tau)$ in a neighborhood of $\tau_0$, which contradicts the maximality of $b$.
Therefore we know that $\tau \to +\infty$ as $t\to b-0$.
Since $\lim_{\tau \to \infty}x(\tau) = x_\ast$, then $x_\ast$ is an equilibrium of (\ref{ODE-desing}) on $\partial {\mathcal{D}}$.
The similar arguments show that $t\to a+0$ corresponds to $\tau \to -\infty$ and that the same consequence holds true.
\end{proof}

This theorem shows that divergent solutions in the direction $x_\ast$ correspond to trajectories of (\ref{ODE-desing}) on the stable manifold $W^s(x_\ast)$\footnote{
The stable set $W^s(p)$ of a point $p$ is characterized as $\{x=x(0)\mid d(x(\tau), p)\to 0 \text{ as }\tau \to \infty\}$ with a metric $d$ on the phase space.
If $p$ is an equilibrium, the {\em (center-)stable manifold theorem} indicates that the set $W^s(p)$ is, at least locally, has a smooth manifold structure, which is called a {\em (local) stable manifold} of $p$.
}
of the equilibrium $x_\ast$.
This correspondence opens the door to applications of various results in dynamical systems to divergent solutions.
Before moving to the blow-up argument, we gather several properties of dynamics at infinity, which will be useful to concrete studies.

\begin{thm}[Dynamics at infinity, cf. \cite{Mat}]
\label{thm-dyn-infty}
\begin{enumerate}
\item The horizon $\mathcal{E} = \partial {\mathcal{D}}$ is an invariant manifold of (\ref{ODE-desing}).
%
%
\item Dynamics of (\ref{ODE-desing}) on $\mathcal{E}$ is dominated by 
\begin{equation*}
\dot x_i = (\tilde f_{\alpha,k})_i - \left(\sum_{j=1}^n \beta_j x^{2\beta_j - 1}(\tilde f_{\alpha,k})_j \right) \frac{x_i}{\beta_i}.
\end{equation*}
\item Time evolution of $1-p(x)^{2c}$ in $\tau$-time scale is dominated by
\begin{equation*}
\frac{d}{d\tau} (1-p(x)^{2c}) =  -\left(\sum_{j=1}^n \beta_j x_j^{2\beta_j-1}\tilde f_j\right) (1-p(x)^{2c}).
\end{equation*}
\item Assume that the vector field $f$ in (\ref{ODE-original}) is quasi-homogeneous of type $(\alpha_1,\cdots, \alpha_n)$ and order $k+1$. Then the desingularized vector field $g$ given in (\ref{ODE-desing}) satisfies
\begin{equation}
\label{symmetry}
g_i((-1)^{\alpha_1} x_1,\cdots, (-1)^{\alpha_n} x_n) = (-1)^{k+\alpha_n} g_i(x_1,\cdots, x_n).
\end{equation}
In particular, for any asymptotically quasi-homogeneous vector field $f$ in (\ref{ODE-original}) of type $(\alpha_1,\cdots, \alpha_n)$ and order $k+1$, the desingularized vector field $g$ satisfies (\ref{symmetry}) on $\mathcal{E}$.
\end{enumerate}
\end{thm}

\begin{proof}
See \cite{Mat}.
\end{proof}

\begin{rem}\rm
Theorem \ref{thm-dyn-infty}-4 shows that the vector field at infinity is equivariant with respect to the symmetry $\iota_\alpha (x)$ defined as
\begin{equation*}
(x_1,\cdots, x_n) \mapsto \iota_\alpha(x) \equiv ((-1)^{\alpha_1} x_1, \cdots, (-1)^{\alpha_n} x_n).
\end{equation*}
on $\mathcal{E}$. 
In particular, if $x\in \mathcal{E}$ is an equilibrium of (\ref{ODE-desing}), then so is $\iota_\alpha(x)$.
In the homogeneous case, the symmetry is just $\iota_\alpha(x) = -x$, as stated in Proposition 2.6 of \cite{EG2006}. 
\end{rem}

%
%
\subsection{Desingularized vector field with quasi-parabolic compactifications}

In the case of quasi-parabolic compactifications, there is an alternative time-variable desingularization given as follows.

\begin{dfn}[Time-variable desingularization for quasi-parabolic compactifications]\rm
Let $y(t)$ be a solution of (\ref{ODE-original}) with an asymptotically quasi-homogeneous vector field $f$ of type $\alpha$ and order $k+1$.
Let also $x = T_{para}(y)$ be the image of $y$ via the quasi-parabolic compactification of type $\alpha$.
Define the new time variable $\tau$ depending on $y(t)$ by
\begin{equation}
\label{time-desing-para}
d\tau = \kappa(y(t))^{k}\left(1-\frac{2c-1}{2c}\kappa^{-1}\right)^{-1} dt = (1-p(x)^{2c})^{-k}\left(1-\frac{2c-1}{2c}(1-p(x)^{2c}) \right)^{-1}dt.
\end{equation}
We shall call (\ref{time-desing-para}) {\em the quasi-parabolic time-variable desingularization of (\ref{vectorfield-cw})}.
\end{dfn}

In the case of quasi-Poincar\'{e} compactifications, the desingularized vector field $g$ associated with the vector field $f$ is {\em not always $C^1$} even if $f$ is sufficiently smooth because of the presence of radicals in $\kappa$ (see \cite{Mat}).
On the other hand, in the case of quasi-parabolic compactifications, if $f$ is smooth, the corresponding desingularized vector field can be {\em always} smooth on $\overline{\mathcal{D}}$ with the alternative time-variable desingularization.
This big difference is one of the reasons why we introduce an alternative quasi-homogeneous compactifications, which is mentioned again in Section \ref{section-blow-up}.

\begin{prop}
\label{prop-para-extension}
Let $f$ be an asymptotically quasi-homogeneous, $C^1$ vector field $f$ of type $\alpha$ and order $k+1$.
Let $x = T_{para}(y)$ be a new variable through quasi-parabolic compactification.
Then the vector field associated with (\ref{vectorfield-cw}) and $\tau$-timescale given in (\ref{time-desing-para}) is $C^1$ on $\overline{\mathcal{D}}$. 
\end{prop}

\begin{proof}
The desingularized vector field in the time variable $\tau$ given in (\ref{time-desing-para}) is
\begin{equation}
\label{ODE-desing-para}
\frac{dx_i}{d\tau} = \left(1-\frac{2c-1}{2c}(1-p(x)^{2c}) \right)\tilde f_i(x) - \alpha_i x_i \sum_{j=1}^n \frac{x_j^{2\beta_j-1}}{\alpha_j}\tilde f_j(x).
\end{equation}
Note that each $\tilde f_j(x)$ is $C^1$ on $\overline{\mathcal{D}}$, since all terms of $\tilde f_j$ are multiples of powers of $(1-p(x)^{2c})$ and smooth asymptotically quasi-homogeneous terms in $f_j(y)$.
Consequently, we know that the right-hand side of (\ref{ODE-desing-para}) is $C^1$ on $\overline{\mathcal{D}}$.
\end{proof}

%
%
\subsection{Desingularized vector field with directional compactifications}

The desingularized vector field associated with $f$ is also considered with directional compactifications like (\ref{dir-cpt}).
For simplicity, set $i=n$ in (\ref{dir-cpt}).
The corresponding dynamics in $t$-timescale is then calculated as follows:
\begin{equation}
\label{vec-directional}
\begin{pmatrix}
y_1' \\ y_2' \\ \vdots \\ y_n'
\end{pmatrix}
=
\begin{pmatrix}
-\alpha_1 s^{-(\alpha_1+1)} x_1 & s^{-\alpha_1}  & 0 & \cdots & 0\\
-\alpha_2 s^{-(\alpha_2+1)} x_2 & 0 & s^{-\alpha_2} & \cdots & 0\\
\vdots & \vdots & \vdots & \ddots & \vdots \\
-\alpha_{n-1} s^{-(\alpha_{n-1}+1)} & 0 & 0 & \cdots & s^{-\alpha_{n-1}}\\
-\alpha_n s^{-(\alpha_n+1)} & 0 &0 &  \cdots & 0
\end{pmatrix}
\begin{pmatrix}
s' \\ x_1' \\ x_2' \\ \vdots \\ x_{n-1}'
\end{pmatrix}
\equiv D_s
\begin{pmatrix}
s' \\ x_1' \\ x_2' \\ \vdots \\ x_{n-1}'
\end{pmatrix}.
\end{equation}
It follows that the matrix $D_s$ is written by the following product of matrices (see also \cite{Mat}):
\begin{equation*}
D_s = \begin{pmatrix}
s^{-\alpha_1} & 0 & \cdots & 0 \\
0 & s^{-\alpha_2} & \cdots & 0\\
\vdots & \vdots & \ddots & \vdots \\
0 & 0 & \cdots & s^{-\alpha_n}
\end{pmatrix}
\begin{pmatrix}
\alpha_1 x_1 & 1 & 0 & \cdots & 0\\
\alpha_2 x_2 & 0 & 1 & \cdots & 0 \\
\vdots & \vdots  & \vdots & \ddots & \vdots \\
\alpha_{n-1} x_{n-1} & 0 & 0 & \cdots & 1\\
\alpha_n & 0 & 0 & \cdots & 0
\end{pmatrix}
\begin{pmatrix}
-s^{-1} & 0 & \cdots & 0 \\
0 & 1 & \cdots & 0\\
\vdots & \vdots & \ddots & \vdots \\
0 & 0 & \cdots & 1
\end{pmatrix}.
\end{equation*}

Since $\alpha_i > 0$ for all $i$, then the matrix $D_s$ is invertible on $\{s>0\}\times \mathbb{R}^n$ to obtain
\begin{equation*}
D_s^{-1} = \begin{pmatrix}
-s & 0 & \cdots & 0 \\
0 & 1 & \cdots & 0\\
\vdots & \vdots & \ddots & \vdots \\
0 & 0 & \cdots & 1
\end{pmatrix}
B
\begin{pmatrix}
s^{\alpha_1} & 0 & \cdots & 0 \\
0 & s^{\alpha_2} & \cdots & 0\\
\vdots & \vdots & \ddots & \vdots \\
0 & 0 & \cdots & s^{\alpha_n}
\end{pmatrix},
\end{equation*}
where $B$ is the inverse\footnote{
The existence of $B$ immediately follows by cyclic permutations and the fact that $\alpha_n > 0$.
}
of the matrix
\begin{equation*}
\begin{pmatrix}
\alpha_1 x_1 & 1 & 0 & \cdots & 0\\
\alpha_2 x_2 & 0 & 1 & \cdots & 0 \\
\vdots & \vdots  & \vdots & \ddots & \vdots \\
\alpha_{n-1} x_{n-1} & 0 & 0 & \cdots & 1\\
\alpha_n & 0 & 0 & \cdots & 0
\end{pmatrix}.
\end{equation*}

Therefore (\ref{vec-directional}) in $\{s>0\}\times \mathbb{R}^{n-1}$ is equivalent to
\begin{equation}
\label{vec-directional-inv}
\begin{pmatrix}
s' \\ x_1' \\ \vdots \\ x_{n-1}'
\end{pmatrix}
=
\begin{pmatrix}
-s & 0 & \cdots & 0 \\
0 & 1 & \cdots & 0\\
\vdots & \vdots & \ddots & \vdots \\
0 & 0 & \cdots & 1
\end{pmatrix}
B
\begin{pmatrix}
s^{\alpha_1} & 0 & \cdots & 0 \\
0 & s^{\alpha_2} & \cdots & 0\\
\vdots & \vdots & \ddots & \vdots \\
0 & 0 & \cdots & s^{\alpha_n}
\end{pmatrix}
\begin{pmatrix}
y_1' \\ y_2' \\ \vdots \\ y_n'
\end{pmatrix}.
\end{equation}
Similarly to (\ref{f-tilde}), let
\begin{equation}
\label{f-tilde-directional}
\hat f_j(s, x_1, \cdots, x_{n-1}) := s^{k+\alpha_j} f_j(s^{-\alpha_1}x_1, \cdots, s^{-\alpha_{n-1}}x_{n-1}, s^{-\alpha_n}),\quad j=1,\cdots, n.
\end{equation}
Then (\ref{vec-directional-inv}) is rewritten as
\begin{equation}
\label{vec-directional-inv-2}
\begin{pmatrix}
s' \\ x_1' \\ \vdots \\ x_{n-1}'
\end{pmatrix}
=
s^{-k}\begin{pmatrix}
-s & 0 & \cdots & 0 \\
0 & 1 & \cdots & 0\\
\vdots & \vdots & \ddots & \vdots \\
0 & 0 & \cdots & 1
\end{pmatrix}
B
\begin{pmatrix}
\hat f_1 \\ \hat f_2 \\ \vdots \\ \hat f_n
\end{pmatrix}.
\end{equation}
The form of $\hat f_i$ in (\ref{f-tilde-directional}) and asymptotic quasi-homogeneity of $f_i$ and $s$-independence of the matrix $B$ immediately yield the following consequence, which is the directional compactifications' analogue of Lemma \ref{lem-order}.

\begin{lem}
\label{lem-order-directional}
The right-hand side of (\ref{vec-directional-inv-2}) is $O(s^{-k})$ as $s\to 0$.
More precisely, the $s$-component of (\ref{vec-directional-inv-2}) is $O(s^{-k+1})$ as $s\to 0$.
\end{lem}

Lemma \ref{lem-order-directional} leads to introduce the following transformation of time variable.

\begin{dfn}[Time-variable desingularization: directional compactification version]\rm 
Define the new time variable $\tau_d$ by
\begin{equation}
\label{time-desing-directional}
d\tau_d = s(t)^{-k} dt
\end{equation}
equivalently,
\begin{equation*}
t - t_0 = \int_{\tau_0}^\tau s(\tau_d)^k d\tau_d,
\end{equation*}
where $\tau_0$ and $t_0$ denote the correspondence of initial times, and $s(\tau_d)$ is the solution trajectory $s(t)$ under the parameter $\tau$.
We shall call (\ref{time-desing-directional}) {\em the time-variable desingularization of (\ref{vec-directional-inv-2}) of order $k+1$}.
\end{dfn}

The vector field (\ref{vec-directional-inv-2}) is then desingularized in $\tau_d$-time scale:
\begin{equation}
\label{ODE-desing-directional}
\begin{pmatrix}
\frac{ds}{d\tau_d} \\ \frac{dx_1}{d\tau_d} \\ \vdots \\ \frac{dx_{n-1}}{d\tau_d}
\end{pmatrix}
=
\begin{pmatrix}
-s & 0 & \cdots & 0 \\
0 & 1 & \cdots & 0\\
\vdots & \vdots & \ddots & \vdots \\
0 & 0 & \cdots & 1
\end{pmatrix}
B
\begin{pmatrix}
\hat f_1 \\ \hat f_2 \\ \vdots \\ \hat f_n
\end{pmatrix}\equiv g_d(s,x_1,\cdots, x_{n-1}).
\end{equation}

In particular, we have the following proposition.
\begin{prop}
\label{prop-ext-directional}
Let $\tau_d$ be the new time variable given by (\ref{time-desing-directional}).
Then the vector field $g_d$ in (\ref{ODE-desing-directional}) is continuous on $\{s\geq 0\}\times \mathbb{R}^{n-1}$.
\end{prop}

{\em Equilibria at infinity} under directional compactifications are then characterized as equilibria for (\ref{ODE-desing-directional}) on the horizon $\mathcal{E} = \{s=0\}$.
Note that $g_d$ is smooth on $\{s\geq 0\}\times \mathbb{R}^{n-1}$ if $f$ is smooth.
In \cite{Mat}, the topological equivalence among desingularized vector fields with quasi-Poincar\'{e} compactifications and with directional compactifications {\em including the horizon} is discussed.
In other words, dynamics of desingularized vector fields around the horizon is topologically identical among these compactifications.
An essence of such a result is the admissibility in the sense of Definition \ref{dfn-qP} for the equivalence, which indicates that the equivalence result is also valid for quasi-parabolic compactifications.

%
%
\section{Blow-up criteria and numerical validation procedure}
\label{section-blow-up}

Theorem \ref{thm-diverge} indicates that divergent solutions are described as trajectories on stable manifolds of equilibria on the horizon $\mathcal{E}$ for (\ref{ODE-desing}).
On the other hand, Theorem \ref{thm-diverge} itself does not distinguish blow-up solutions from divergent solutions.
Under additional assumptions to equilibria on $\mathcal{E}$, we can characterize blow-up solutions from the viewpoint of dynamical systems.
In this section, we firstly review a criterion of blow-ups discussed in \cite{Mat}.
Then we provide a methodology for explicit estimates of maximal existence time $t_{\max}$.
Finally, we give an algorithm for validating blow-up solutions with computer assistance.

%
%
\subsection{Blow-up criterion}
Firstly we review an abstract result of blow-up criterion via quasi-homogeneous compactifications.
For a squared matrix $A$, ${\rm Spec}(A)$ denotes the set of eigenvalues of $A$.

\begin{prop}[Stationary blow-up, \cite{Mat}]
\label{prop-blowup}
Assume that (\ref{ODE-original}) has an equilibrium at infinity in the direction $x_\ast$.
Suppose that the desingularized vector field $g$ in (\ref{ODE-desing}) is $C^1$ on $\overline{\mathcal{D}}$, and that $x_\ast$ is hyperbolic for (\ref{ODE-desing}); namely all elements in ${\rm Spec}(Dg(x_\ast))$ are away from the imaginary axis.
Then the solution $y(t)$ of (\ref{ODE-original}) whose image $x = T(y)$ is on $W^s(x_\ast)$ in the desingularized vector field (\ref{ODE-desing}) satisfies $t_{\max} < \infty$; namely, $y(t)$ is a blow-up solution.
Moreover,
\begin{equation*}
p(y(t)) \sim c(t_{\max} - t)^{-1/k}\quad \text{ as }\quad t\to t_{\max},
\end{equation*}
where $k+1$ is the order of asymptotically quasi-homogeneous vector field $f$.
Finally, if the $i$-th component $(x_\ast)_i$ of $x_\ast$ is not zero, then we also have
\begin{equation*}
y_i(t) \sim c(t_{\max} - t)^{-\alpha_i /k}\quad \text{ as }\quad t\to t_{\max}.
\end{equation*}
\end{prop}

Proposition \ref{prop-blowup} gives us the slogan that {\em hyperbolic equilibria at infinity induce blow-up solutions}.
In the above original version of blow-up criterion stated as above, the $C^1$-smoothness of the desingularized vector field $g$ in (\ref{ODE-desing}) is {\em assumed}, because such a smoothness is nontrivial for quasi-Poincar\'{e} compactifications even if $f$ is sufficiently smooth. 
On the other hand, Proposition \ref{prop-para-extension} shows that stability analysis of equilibria at infinity always makes sense with quasi-parabolic compactifications, because in which case the desingularized vector field $g$ is always $C^1$ on $\overline{\mathcal{D}}$ if $f$ is $C^1$.
\par
Needless to say, the above proposition does not provide information of concrete blow-up time $t_{\max}$ depending on initial data.
In the successive subsections, we provide a validation procedure of blow-up solutions with explicit estimates of blow-up times.

%
%
\subsection{Lyapunov functions around asymptotically stable equilibria}
Our main tool for validating blow-up time is {\em Lyapunov function}, which describes the monotonous behavior of trajectories in terms of its value.
As the general setting, consider the vector field
\begin{equation}
\label{ODE}
\frac{dx}{dt} = f(x),\quad f :\mathbb{R}^n\to \mathbb{R}^n\text{: smooth}.
\end{equation}
For $x\in \mathbb{R}^n$, $Df(x)$ denotes the Jacobian matrix of $f$ at $x$.

\begin{prop}[Lyapunov function for stable equilibria, \cite{MHY2016}]
\label{prop-Lyap}
Let $x_\ast$ be an equilibrium for (\ref{ODE}) in a compact star-shaped set $N \subset \mathbb{R}^n$.
Assume that there is a real symmetric matrix $Y$ such that the matrix
\begin{equation}
\label{neg-def-1}
A(x) := Df(x)^TY + YDf(x)
\end{equation}
is strictly negative definite for all $x\in N$.
Then the functional $L:\mathbb{R}^n\to \mathbb{R}$ given by
\begin{equation}
\label{Lyap-present}
L(x) := (x-x_\ast)^TY(x-x_\ast)
\end{equation}
is a Lyapunov function on $N$ such that $dL/dt$ vanishes at $x_\ast$.
In particular, $x_\ast$ is the unique equilibrium in $N$.
If further the matrix $Y$ is chosen to be positive definite, then the equilibrium $x_\ast$ is asymptotically stable.
\end{prop}
We shall call the compact set $N$ satisfying the assumption in Proposition \ref{prop-Lyap} a {\em Lyapunov domain} of $x_\ast$.

\begin{rem}[The present choice of $L(x)$]\rm
\label{rem-Lyap}
Roughly speaking, the matrix $Y$ contains information of sign of the real part of each ${\rm Spec}(Df(x))$ and a matrix representing change of coordinates.
In the present case, we only treat {\em asymptotically stable} equilibria, which indicates that signs of ${\rm Re}\lambda$ for any $\lambda\in {\rm Spec}(Df(x))$ should be identically negative.
Before validating an equilibrium $x_\ast$, it should be usually computed {\em in a numerical (i.e., non-rigorous computation) sense} with associated eigenvalues for finding candidates of validating equilibrium.
\par
When we numerically compute eigenvalues of a Jacobian matrix, say $Df(x)$, we also compute eigenvectors to construct the eigenmatrix $X$, which represents change of coordinates to an {\em orthogonal} one. 
In \cite{MHY2016}, the matrix $Y$ in (\ref{Lyap-present}) is typically defined as $Y = {\rm Re}(X^{-H}X^{-1})$, where $X^{-H} := (X^{-1})^H$ and $\ast^H$ denotes the Hermitian transpose of the object (vectors or matrices).
Note that, in which case, the equilibrium $x_\ast$ is shown to be asymptotically stable in $N$.
However, there are cases that an eigenvalue has multiplicity larger than $1$, in which cases the validation is failed because the computed eigenmatrix $X$ typically becomes singular.
Indeed, our example below contains such a case.
\par
One way to avoid such difficulty is to use the {\em real Schur decomposition} of the matrix $Df(x)$ instead of eigenpair computations.
See Appendix \ref{appendix-Schur} about a quick review of Schur decompositions of matrices.
Let $Q$ be a matrix such that $Q^TDf(x)Q$ is a real upper triangle matrix for some point $x$.
Then we can check the sign of ${\rm Re}\lambda$ for all $\lambda \in {\rm Spec}(Df(x))$.
We then choose the matrix $Q$ as a change of coordinates instead of the eigenmatrix $X$.
In such a case, the corresponding matrix $Y$ is $Y={\rm Re}(Q^{-H}Q^{-1})$.
When we use the real Schur decomposition, $Q$ is an orthogonal real matrix. 
Then we take $Y = I$: the identity matrix, which shows that our Lyapunov function $L$ becomes $L(x) = \|x-x_\ast\|^2$. 
This fact also shows that $x_\ast$ is asymptotically stable in $N$.
\end{rem}

Once we have validated a Lyapunov function $L$ as well as the Lyapunov domain $\tilde N$ of an {\em asymptotically stable} equilibrium $x_\ast$, we can easily characterize global trajectory asymptotic to $x_\ast$.
For a positive number $\epsilon > 0$, assume that $N := \{x\in \mathbb{R}^n\mid L(x) \leq \epsilon \}\subset \tilde N$.
Let $\{x(t)\}_{t\in [0,t_N]}$ be a trajectory of vector field (\ref{ODE}) for some $t_N > 0$ and 
assume that $x(t_N) \in {\rm int}\,N$.
Then the trajectory $x(t)$ behaves {\em so that it strictly decreases $L$}.
Since $N =  \{x\in \mathbb{R}^n \mid L(x) \leq \epsilon \}$, then the trajectory can be continued until it tend to a point on $\{L=0\}$, in which case $x=x_\ast$.
Therefore the trajectory $\{x(t)\}_{t\in [0,t_N]}$ is extended to the global trajectory $\{x(t)\}_{t\in [0,\infty)}$ satisfying $x(t)\to x_\ast$ as $t\to \infty$, as desired.

%
%
\subsection{Explicit estimate of blow-up time with computer assistance}
\label{section-explicit}
Here we provide an explicit estimate methodology of blow-up times.
The basic idea is {\em Lyapunov tracing} discussed in \cite{MHY2016, TMSTMO}; namely, computation of the maximal existence time
\begin{equation*}
t_{\max} = \int_0^\infty \frac{d\tau}{\kappa(T^{-1}(x(\tau)))^k},\quad \text{ or }\quad
t_{\max} = \int_0^\infty s(\tau)^k d\tau,
\end{equation*}
of trajectory $\{y(t) = T^{-1}(x(t))\}$ in terms of Lyapunov functions around an equilibrium $x_\ast$ on the horizon $\mathcal{E}$.
Theorem \ref{prop-blowup} shows that blow-up solutions correspond to trajectories on {\em stable manifolds of hyperbolic equilibria on $\mathcal{E}$}.
According to this fact and preceding methodology in \cite{TMSTMO}, we validate asymptotic behavior of blow-up solutions by the following steps.
A quasi-homogeneous compactification $T:\mathbb{R}^n\to \mathcal{D}$ and time-variable desingularization are assumed to be given in advance.
\begin{enumerate}
\item Validate an equilibrium $x_\ast \in \mathcal{E}$.
\item Validate a Lyapunov function of the form (\ref{Lyap-present}) around $x_\ast$ as well as its Lyapunov domain $\tilde N$.
\end{enumerate}

Now we are ready to validate blow-up times with computer assistance.
Let $T$ be an admissible quasi-homogeneous compactification with type $\alpha$.
Assume that the desingularized vector field (\ref{ODE-desing}) is $C^1$ on $\overline{\mathcal{D}}$, which is always the case when $T = T_{para}$ and $f$ is $C^1$.
Let $x_\ast \in \mathcal{E}$ be an equilibrium on the horizon for (\ref{ODE-desing}).
Explicit estimates of maximal existence time $t_{\max}$ actually depend on the choice of compactifications $T$ and time-variable desingularizations.
In what follows we fix $T$ as the quasi-parabolic compactification $T_{para}$ (associated with type $\alpha$) and quasi-parabolic time-variable desingularization (\ref{time-desing-para}).
\par
Assume that we have computed the global trajectory $\{x(\tau)\}_{\tau \in [0,\infty)}$ for (\ref{ODE-desing-para}) such that $x(\tau)\in N = \{x\in \mathcal{D}\mid L(x)\leq \epsilon\}\subset \tilde N$ for all $\tau\in [\tau_N,\infty)$ and some $\epsilon > 0$, where $\tau_N > 0$ and $\tilde N$ is a Lyapunov domain of an asymptotically stable equilibrium $x_\ast\in \mathcal{E}$\footnote{
In this case, the set $N$ is contained in the stable manifold $W^s(x_\ast)$ of $x_\ast$.
}.
The maximal existence time of $x(\tau)$ in $t$-timescale is then
\begin{align*}
\notag
t_{\max} &= t_N + \int_{\tau_N}^\infty \left(1-\frac{2c-1}{2c}\kappa^{-1}\right)\frac{d\tau}{\kappa(T^{-1}(x(\tau)))^k}\\
\label{tmax-para}
	 &=  t_N + \int_{\tau_N}^\infty \left(1-\frac{2c-1}{2c}(1-p(x(\tau))^{2c})\right)(1-p(x(\tau))^{2c})^k d\tau,
\end{align*}
where
\begin{equation}
\label{t_N}
t_N = \int_0^{\tau_N} \left(1-\frac{2c-1}{2c}(1-p(x(\tau))^{2c})\right)(1-p(x(\tau))^{2c})^k d\tau.
\end{equation}
Then compute an upper bound of $t_{\max}$ by
\begin{equation}
\label{tmax-estimate}
0 < t_{\max} - t_N \leq  \frac{1}{c_{\tilde N}c_1} \int_0^{L(x(\tau_N))} \frac{C_{n,\alpha,N}(L)^{k}}{L}dL
\leq \frac{1}{c_{\tilde N}c_1} \int_0^{\epsilon} \frac{C_{n,\alpha,N}(L)^{k}}{L}dL \equiv \overline{C_{n,\alpha,k, N}}(\epsilon),
\end{equation}
where $L = L(x)$ is the value of validated Lyapunov function at $x\in N$,  
$c_1$ and $c_{\tilde N}$ are constants involving eigenvalues of $Y$ and $A(x)$ whose details are shown in \cite{TMSTMO}.
This inequality comes from the property of Lyapunov function following the definition:
\begin{equation*}
\frac{dL}{d\tau}(x(\tau))_{\tau = 0} \leq -c_1 c_{\tilde N}L(x(0)),
\end{equation*}
which is strictly negative as long as $x(0)\not = x_\ast$.
See \cite{TMSTMO} for the detail.
A function $C_{n,\alpha,N}(L)$ depends on the value $L$ of Lyapunov function satisfying
\[
	\left|1-p(x)^{2c}\right|\le C_{n,\alpha,N}(L)\quad \text{ for }\quad x\in \tilde N.
\]
Concrete estimates of the function $C_{n,\alpha, N}(L)$ we have used in practical validations are derived in Appendix \ref{appendix-estimates}.
Since $L(x(\tau_N)) \leq \epsilon$, the rightmost side of (\ref{tmax-estimate}) is an integral {\em on a compact interval}.
If we can estimate the right-hand side of (\ref{tmax-estimate}) being finite, we obtain a finite upper bound of $t_{\max}$, which shows that the trajectory $\{y(t)\}_{t\in [0,t_{\max})} = \{T^{-1}(x(\tau))\}_{\tau\in [0,\infty)}$ is a blow-up solution of the original initial value problem (\ref{ODE-original}) with blow-up time $t_{\max}\in [t_N, t_N + \overline{C_{n,\alpha,k,N}}(\epsilon)]$.

\par
\bigskip
The similar estimate is derived in the case of directional compactifications.
In such a case with the same setting as above, the maximal existence time $t_{\max}$ is computed as
\begin{equation*}
t_{\max} = \int_0^\infty s(\tau_d)^k d\tau_d = t_N + \int_{\tau_N}^\infty s(\tau_d)^k d\tau_d,
\end{equation*}
where $t_N = \int_0^{\tau_N} s(\tau_d)^k d\tau_d$.
Assume that the trajectory $\{(s(\tau), x(\tau))\}_{\tau\in [0,\tau_N]}$ enters inside ${\rm int}N:= \{L(s,x) < \epsilon\} \subset \tilde N$, where $\tilde N$ is a Lyapunov domain of $(0,x_\ast) \in \mathcal{E}$.
Then we have
\begin{align}
\int_{\tau_N}^\infty s(\tau_d)^k d\tau_d & \leq \int_{\tau_N}^\infty (|s|^2 + \|x-x_\ast\|^2 )^k d\tau_d\nonumber\\
	&\leq \int_{\tau_N}^\infty \left\{ c_1 L(s(\tau),x(\tau)) \right\}^{k/2} d\tau_d\nonumber\\
	&\leq - \int_{L(s(\tau_N),x(\tau_N))}^0 \left\{ c_1 L \right\}^{k/2} \frac{dL}{\tilde c_N c_1 L}\nonumber\\
	&=  \frac{1}{\tilde c_N \sqrt{c_1}} \int_0^{L(s(\tau_N),x(\tau_N))} L^{\frac{k}{2}-1} dL\nonumber\\
	&\leq  \frac{1}{\tilde c_N \sqrt{c_1}} \left[ 2L^{k/2} \right]_0^{\epsilon} = \frac{2}{\tilde c_N} \sqrt{\frac{\epsilon^k}{c_1}} \equiv \overline{C_{n,k,N}}(\epsilon).\label{tmax-estimate-dir}
\end{align}
The rightmost quantity gives an upper bound of $t_{\max}$.
More precisely, the blow-up time $t_{\max}$ is a value in $[t_N, t_N + \overline{C_{n,k,N}}(\epsilon)]$.

%
%
\subsection{Validation procedure of blow-up solutions}

Now we have obtained an explicit estimate of blow-up times.
Theorem \ref{thm-diverge} indicates that blow-up solutions correspond to trajectories on stable manifolds of (hyperbolic) equilibria at infinity, which can be validated by standard numerical validation techniques of dynamical systems (e.g., \cite{kv}).

Our algorithm for validating blow-up solutions is the following, which is essentially same as that in the preceding work \cite{TMSTMO}.

\begin{alg}[Validation of blow-up solutions with quasi-parabolic compactifications]
\label{alg-validation1}
Let $f:\mathbb{R}^n\to \mathbb{R}^n$ be an asymptotically quasi-homogeneous, smooth vector field of type $\alpha = (\alpha_1,\cdots, \alpha_n)$ and order $k+1$.
Choose natural numbers $\beta_1,\cdots, \beta_n, c\in \mathbb{N}$ so that (\ref{LCM}) holds.
Let $T_{para}:\mathbb{R}^n\to \mathcal{D}$ be a quasi-parabolic compactification and $\frac{dx}{d\tau} = g(x)$ be the associated desingularized vector field with time-variable desingularization (\ref{time-desing-para}).
\begin{enumerate}
\item Validate an equilibrium at infinity $x_\ast$; namely, a zero of $g$ on $\mathcal{E} = \partial \mathcal{D}$.
\item Construct a compact, star-shaped set $\tilde N\subset \overline{\mathcal{D}}$ containing $x_\ast$ so that the negative definiteness of (\ref{neg-def-1}) on $\tilde N$ with a positive definite, real symmetric matrix $Y$ is validated as large as possible.
If we cannot find such a set $\tilde N$, return {\tt failed}.
\item 
Let $L(x)=(x-x_\ast)^T Y (x-x_\ast)$ be the validated Lyapunov function on $\tilde N$. 
Set $\epsilon > 0$ as the maximal value so that $N := \{x\in \mathbb{R}^n \mid L(x)\leq \epsilon\} \subset \tilde N$.
Integrate the ODE $(dx/d\tau) = g(x)$ with initial data $x_0\in \mathcal{D}$ until $\tau = \tau_N$ so that $x(\tau_N)\in {\rm int}N$.
If we cannot find such $x(\tau_N)$, return {\tt failed}.
\item Compute $\overline{C_{n,\alpha,k, N}}(\epsilon)$. 
Simultaneously, compute $t_N$ following (\ref{t_N}).
If $\overline{C_{n,\alpha,k, N}}(\epsilon)$ can be validated to be finite, return {\tt succeeded}.
\end{enumerate}
\end{alg}

The similar algorithm with directional compactifications is derived as follows.

\begin{alg}[Validation of blow-up solutions with directional compactifications]
\label{alg-validation2}
Let $f:\mathbb{R}^n\to \mathbb{R}^n$ be an asymptotically quasi-homogeneous, smooth vector field of type $\alpha = (\alpha_1,\cdots, \alpha_n)$ and order $k+1$.
Let $T:\mathbb{R}^n\to \{s>0\}\times \mathbb{R}^{n-1}$ be a directional compactification determined by (\ref{dir-cpt}) and $\frac{d(s,x)}{d\tau_d} = g_d(s,x)$ be the associated desingularized vector field with time-variable desingularization (\ref{time-desing}).
\begin{enumerate}
\item Validate an equilibrium at infinity $(0,x_\ast)$; namely, a zero of $g_d$ on $\mathcal{E} = \{s=0\}$.
\item Construct a compact, star-shaped set $\tilde N\subset \{s\geq 0\}\times \mathbb{R}^{n-1}$  containing $(0, x_\ast)$ so that the negative definiteness of (\ref{neg-def-1}) on $\tilde N$ with a positive definite, real symmetric matrix $Y$ is validated as large as possible.
If we cannot find such a set $\tilde N$, return {\tt failed}.
\item 
Let $L(s,x)= ((s,x)- (0,x_\ast))^T Y((s,x)- (0,x_\ast))$ be the validated Lyapunov function on $\tilde N$. 
Set $\epsilon > 0$ as the maximal value so that $N := \{(s,x)\in \{s\geq 0\}\times \mathbb{R}^{n-1} \mid L(s,x)\leq \epsilon\} \subset \tilde N$.
Integrate the ODE $(d(s,x)/d\tau_d) = g_d(s,x)$ with initial data $(s_0,x_0)\in \{s> 0\}\times \mathbb{R}^{n-1}$ until $\tau = \tau_N$ so that $(s(\tau_N), x(\tau_N))\in {\rm int}N$.
If we cannot find such $(s(\tau_N), x(\tau_N))$, return {\tt failed}.
\item Compute $\overline{C_{n,k, N}}(\epsilon)$. 
Simultaneously, compute $t_N = \int_0^{\tau_N} s(\tau_d)^k d\tau_d$.
If $\overline{C_{n,k, N}}(\epsilon)$ can be validated to be finite, return {\tt succeeded}.
\end{enumerate}
\end{alg}

Under the successful operations of Algorithm \ref{alg-validation1} or \ref{alg-validation2}, we have the following results, which show the validation of blow-up solutions.
The proofs immediately follow from properties of compactifications and Lyapunov functions.

\begin{thm}[Validation of blow-up solutions with quasi-parabolic compactifications]
\label{thm-blowup-validate}
Let $y_0\in \mathbb{R}^n$.
Assume that Algorithm \ref{alg-validation1} returns {\tt succeded} with $x_0 = T(y_0)$.
Then the solution $\{y(t) = T^{-1}(x(t))\}$ of (\ref{ODE-original}) with $y(0) = y_0$ such that 
\begin{equation*}
\{x(\tau)\mid \tau \in [0, \infty), x(\tau)\to x_\ast \text{ as }\tau\to \infty\}
\end{equation*}
via a time-variable desingularization (\ref{time-desing}) and an asymptotically stable equilibrium $x_\ast \in \mathcal{E}$ is a blow-up solution with the blow-up time $t_{\max} \in [\tau_N, \tau_N + \overline{C_{n,\alpha,k, N}}(\epsilon)]$.
\end{thm}

\begin{thm}[Validation of blow-up solutions with directional compactifications]
\label{thm-blowup-validate-directional}
Let $y_0\in \mathbb{R}^n$.
Assume that Algorithm \ref{alg-validation2} returns {\tt succeded} with $(s_0, x_0) = T(y_0)$.
Then the solution $\{y(t) = T^{-1}(s(t), x(t))\}$ of (\ref{ODE-original}) with $y(0) = y_0$ such that 
\begin{equation*}
\{(s(\tau_d), x(\tau_d))\mid \tau_d \in [0, \infty), s(\tau_d)\to 0,\ x(\tau_d)\to x_\ast \text{ as }\tau_d\to \infty\}
\end{equation*}
via a time-variable desingularization (\ref{time-desing-directional}) and an asymptotically stable equilibrium $(0,x_\ast) \in \mathcal{E}$ is a blow-up solution with the blow-up time $t_{\max} \in [\tau_N, \tau_N + \overline{C_{n,k, N}}(\epsilon)]$.
\end{thm}

Finally we remark that our validations do {\em not} contain those of hyperbolicity for equilibria on $\mathcal{E}$.
Indeed, we only verify negative definiteness of {\em the symmetrization} of $Df(x)$ or an associated matrix in (\ref{neg-def-1}).
In particular, our validation does not directly provide rigorous blow-up rates of blow-up solutions mentioned in Proposition \ref{prop-blowup}.
Nevertheless, Proposition \ref{prop-blowup} provides a guideline for focusing on our targeting objects for validations, and Lyapunov function validations yield the asymptotic stability of equilibria and rigorous estimates of blow-up times, as mentioned.

%
%
\section{Validation examples}
\label{section-numerical}
In this section, we demonstrate our procedure with several test problems.
All computations were carried out on 
macOS Sierra (ver.~10.12.5),
Intel(R) Xeon(R) CPU E5-1680 v2 @ 3.00 GHz using the kv library \cite{kv} ver. 0.4.41 to rigorously compute the trajectories of ODEs.

%
%
\subsection{Example 1}
\label{section-ex1}
The first example is the following two-dimensional ODE:
\begin{equation}
\label{KK-simple}
\begin{cases}
u' = u^2 - v, & \\
v' = \frac{1}{3}u^3. &
\end{cases}
\end{equation}
This vector field is the special case of (\ref{KK}) discussed in the next example.
It immediately holds that the vector field (\ref{KK-simple}) is quasi-homogeneous of type $(1,2)$ and order $2$.
The numerical study of complete dynamics including infinity is shown in \cite{Mat}.
Our purpose here is to validate a blow-up solution observed there.
We introduce the quasi-parabolic compactification of type $(1,2)$ given by
\begin{equation*}
u = \frac{x_1}{1-p(x)^4},\quad v = \frac{x_2}{(1-p(x)^4)^2},\quad p(x)^4 = x_1^4 + x_2^2.
\end{equation*}
Then the corresponding desingularized vector field (\ref{ODE-desing-para}) is given by the following: 
\begin{equation}
\label{KK-simple-desing}
\begin{cases}
\dot x_1 = (x_1^2 - x_2) F(x) - x_1G(x) & \\
\dot x_2 = \frac{1}{3}x_1^3 F(x) - 2x_2 G(x) &
\end{cases},\quad \dot {} = \frac{d}{d\tau},
\end{equation}
where
\begin{equation*}
F(x) = \frac{1}{4}\left\{ 1 + 3(1-p(x)^4)\right\},\quad G(x) = x_1^3 (x_1^2 - x_2) + \frac{1}{6}x_1^3 x_2.
\end{equation*}
We are then ready to validate blow-up solutions, following Algorithm \ref{alg-validation1}.
In the similar way to \cite{Mat}, it turns out that the system (\ref{KK-simple-desing}) admits exactly four equilibria at infinity, one of which is a sink\footnote{
An equilibrium $p$ with ${\rm Spec}(Dg(p)) \subset \{\lambda\in \mathbb{C}\mid {\rm Re}\lambda < 0\}$.
}, the other one of which is a source\footnote{
An equilibrium $p$ with ${\rm Spec}(Dg(p)) \subset \{\lambda\in \mathbb{C}\mid {\rm Re}\lambda > 0\}$
}
and the rest of two are saddles\footnote{
Hyperbolic equilibria which are not neither sinks nor sources.
}.
Here we compute the sink on the horizon satisfying
\[
	x_\ast\in\left(
	\begin{array}{c}
	\left[0.98913699589497727,0.98913699589497773\right]\\
	\left[0.20675855700518036,0.2067585570051809\right]
	\end{array}\right),
\]
where $[\cdot,\cdot]$ denotes a real interval.
After that we validate a Lyapunov function as well as its Lyapunov domain including $N= \{x\in \overline{\mathcal{D}}\mid L(x) \leq \epsilon \}$ around the sink and  a solution trajectory $x(\tau)$ which enters $N$ in a finite time $\tau_N$.
The initial data are given by $(x_1(0),x_2(0))=(-0.1,0.0001)$ and $(-0.1,-0.1)$.
Table \ref{Tab:Ex1} shows validated results of blow-up solutions for (\ref{KK-simple}).
See also Figure \ref{fig-Ex1}.

\begin{table}[ht]
\caption{Validated results for (\ref{KK-simple}): numerical validations prove $x(\tau_N)\in {\rm int}\,N$ and (\ref{tmax-estimate}) yields the inclusion of the blow-up time $t_{\max}$. Subscript and superscript numbers in the table denote lower and upper bounds of the interval, respectively.}
\centering
\resizebox{\textwidth}{!}{
\begin{tabular}{ccccc}
\hline
Initial data & $\epsilon$ & $\tau_N$ & Inclusion of $t_{\max}$ & Exec. time\\
\hline
$(-0.1,0.0001)$ & $5.6700023252180213\times 10^{-5}$ & $343.57935744230372$ & $84.083_{706663650346}^{853417007874}$ & 1.42 s\\
$(-0.1,-0.1)$ & $5.6700023252180213\times 10^{-5}$ & $32.05598188250481$ & $6.201_{0761835235443}^{2442938861261}$ & 1.11 s\\
\hline 
\end{tabular}%
}
\label{Tab:Ex1}
\end{table}%

\begin{figure}[htbp]\em
\begin{minipage}{0.5\hsize}
\centering
\includegraphics[width=7.0cm]{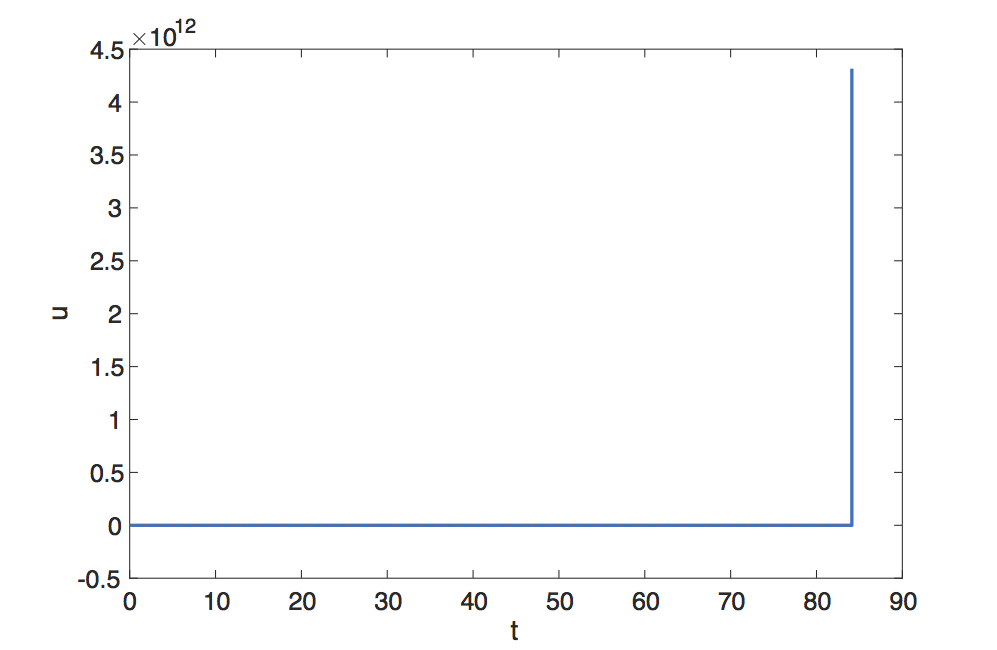}
(a)
\end{minipage}
\begin{minipage}{0.5\hsize}
\centering
\includegraphics[width=7.0cm]{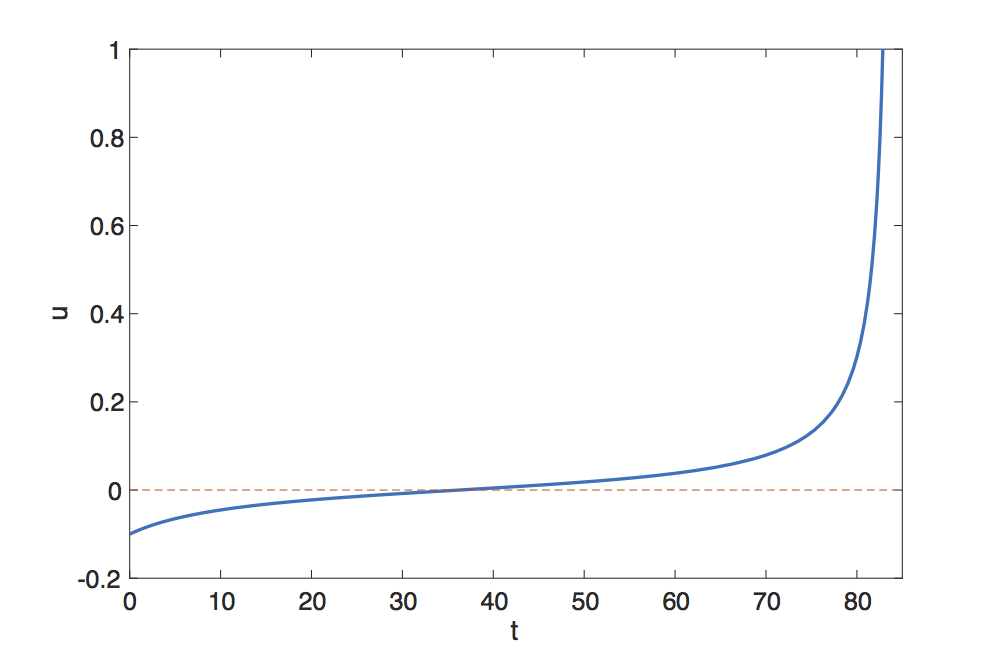}
(b)
\end{minipage}\\
\begin{minipage}{0.5\hsize}
\centering
\includegraphics[width=7.0cm]{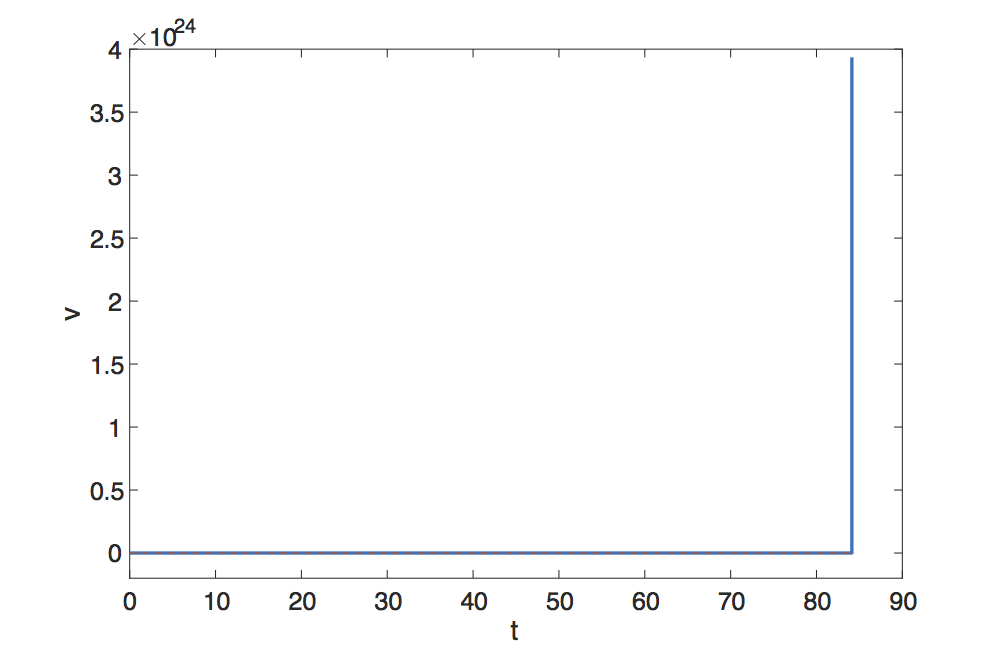}
(c)
\end{minipage}
\begin{minipage}{0.5\hsize}
\centering
\includegraphics[width=7.0cm]{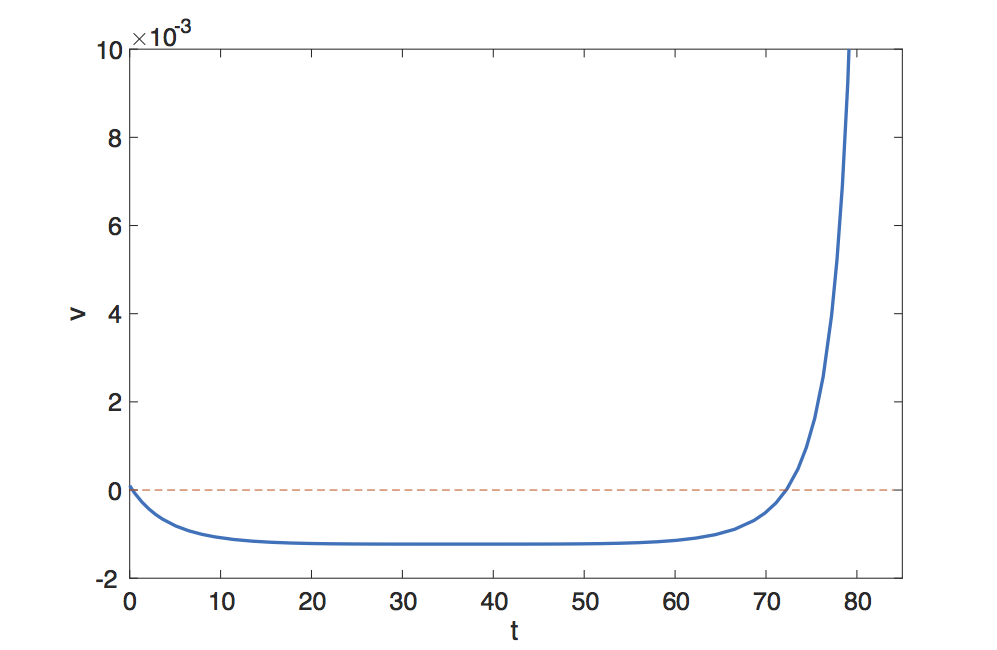}
(d)
\end{minipage}\\
\caption{A blow-up trajectory for (\ref{KK-simple})}
\label{fig-Ex1}
A blow-up trajectory with the initial data $(u(0),v(0)) = T_{para}^{-1}(x)$, $(x_1(0), x_2(0)) = (-0.1, 0.0001)$ are drawn.
Horizontal axis is the original time variable $t$, and vertical axis is the value of variables $u$ and $v$.
(a) : the $u$-component of the blow-up trajectory. (b) : the $u$-component of the blow-up trajectory in a vicinity of $u=0$. 
(c) : the $v$-component of the blow-up trajectory. (d) : the $v$-component of the blow-up trajectory in a vicinity of $v=0$.  
\end{figure}

Validated results in this example show the efficiency of quasi-parabolic compactifications compared with directional compactifications.
As indicated in \cite{Mat} and Figure \ref{fig-Ex1}, validated trajectories can change the sign.
Numerical computations as well as rigorous validations of trajectories with directional compactifications require the assumption that (at least) one of components {\em never change the sign} and that, even if it is the case, one knows such a component in advance.
If we deal with sign-changing trajectories, coordinate-change transformations have to be incorporated into the whole computations, which are not easy tasks for numerical integration of differential equations.
On the other hand, there is no such worry with quasi-parabolic compactifications because they provide {\em globally defined charts} on embedding manifolds.
Trajectories can be therefore validated without any assumptions about their signs.

%
%
\subsection{Example 2}
\label{section-ex2}
The second example is the following two-dimensional ODE:
\begin{equation}
\label{KK}
\begin{cases}
u' = u^2 - v -su - c_1, & \\
v' = \frac{1}{3}u^3 - u - sv - c_2, &
\end{cases}
\end{equation}
where $(c_1,c_2) = (c_{1L}, c_{2L})$ or $(c_{1R}, c_{2R})$ are constants with
\begin{equation*}
\begin{cases}
c_{1L} = u_L^2 - v_L - su_L, & \\
c_{2L} = \frac{1}{3}u_L^3 - u_L - sv_L, & \\
\end{cases}
\quad
\begin{cases}
c_{1R} = u_R^2 - su_R - v_R, & \\
c_{2R} = \frac{1}{3}u_R^3 - u_R - sv_R. & \\
\end{cases}
\end{equation*}
The system (\ref{KK}) is well-known as the traveling wave equation derived from the {\em Keyfitz-Kranser model} \cite{KK1990}, which is the following initial value problem of the system of conversation laws:
\begin{equation}
\label{KK-PDE}
\begin{cases}
\displaystyle
{
\frac{\partial u}{\partial t} + \frac{\partial }{\partial x}(u^2 - v) = 0,
} & \\
\displaystyle
{
\frac{\partial v}{\partial t} + \frac{\partial }{\partial x}\left(\frac{1}{3}u^3 - u\right)=0,
} &
\end{cases}
\quad 
(u(x,0), v(x,0)) = 
\begin{cases}
(u_L, v_L) & x<0, \\
(u_R, v_R) & x>0.
\end{cases}
\end{equation}
In particular, our attentions are restricted to solutions of the form
\begin{equation}
\label{KK-ansatz}
u(x,t) = \bar u(\xi),\quad v(x,t) = \bar v(\xi),\quad \xi = x-st
\end{equation}
satisfying the following boundary condition:
\begin{equation}
\label{RH}
\lim_{\xi \to -\infty} \begin{pmatrix} \bar u(\xi) \\ \bar v(\xi) \end{pmatrix} = \begin{pmatrix} u_L \\ v_L \end{pmatrix},\quad \lim_{\xi \to +\infty} \begin{pmatrix} \bar u(\xi) \\ \bar v(\xi) \end{pmatrix} = \begin{pmatrix} u_R \\ v_R \end{pmatrix}.
\end{equation}

The governing system with the ansatz (\ref{KK-ansatz})-(\ref{RH}) derives the system (\ref{KK}).

\begin{rem}\rm
The system (\ref{KK-simple}) in the previous example actually extracts the quasi-homogeneous part of (\ref{KK}).
\par
Solutions (\ref{KK-ansatz}) of (\ref{KK}) satisfying (\ref{RH}) correspond to {\em shock waves with speed $s$} for the Riemann (initial value) problem (\ref{KK-PDE}) satisfying {\em viscosity profile criterion}.
The boundary condition (\ref{RH}) is known as the {\em Rankine-Hugoniot condition} which weak solutions of (\ref{KK-PDE}) admitting discontinuity must be satisfied.
\par
On the other hand, it is well-known that the Riemann problem (\ref{KK-PDE}) admits shock wave solutions with Dirac-delta singularities called {\em singular shock waves}.
Such solutions satisfy only a part of (\ref{RH}); called {\em Rankine-Hugoniot deficit}, and such structure corresponds to the presence of blow-up solutions for (\ref{KK}) with $(c_1,c_2) = (c_{1L}, c_{2L})$ or $(c_{1R}, c_{2R})$, which inspires our considerations herein.
See e.g., \cite{KK1990, SSS1993} for details about (\ref{KK-PDE}).
\end{rem}

It immediately holds that, as in the previous example, the vector field (\ref{KK}) turns out to be {\em asymptotically} quasi-homogeneous at infinity with type $\alpha = (1,2)$ and order $2$.
The desingularized vector field with quasi-parabolic compactifications is calculated as follows.
Introduce the quasi-parabolic compactification of type $(1,2)$ given by
\begin{equation*}
u = \frac{x_1}{1-p(x)^4},\quad v = \frac{x_2}{(1-p(x)^4)^2},\quad p(x)^4 = x_1^4 + x_2^2
\end{equation*}
and nonlinear functions $\tilde f_1(x), \tilde f_2(x)$ by
\begin{equation*}
\tilde f_1(x) := x_1^2 - x_2 - s\kappa^{-1}x_1 - \kappa^{-2}c_1,\quad \tilde f_2(x) := \frac{1}{3}x_1^3 - \kappa^{-2}x_1 - s\kappa^{-1}x_2 - c_2\kappa^{-3},
\end{equation*}
where $\kappa^{-1} = \kappa(x)^{-1} = (1-p(x)^4)^{1/4}$, the desingularized vector field associated with (\ref{KK}) becomes
\begin{equation*}
\label{KK-desing}
\begin{cases}
\dot x_1 = (x_1^2 - x_2 - s\kappa^{-1}x_1 - \kappa^{-2}c_1) F(x) - x_1 \tilde G(x) & \\
\dot x_2 = \left(\frac{1}{3}x_1^3 - \kappa^{-2}x_1 - s\kappa^{-1}x_2 - c_2\kappa^{-3}\right) F(x) - 2x_2 \tilde G(x) &
\end{cases},\quad \dot {} = \frac{d}{d\tau},
\end{equation*}
where
\begin{align*}
F(x) &= \frac{1}{4}\left\{ 1 + 3(1-p(x)^4)\right\},\\
\tilde G(x) &= x_1^3 (x_1^2 - x_2 - s\kappa^{-1}x_1 - \kappa^{-2}c_1) + \frac{1}{2} x_2\left(\frac{1}{3}x_1^3 - \kappa^{-2}x_1 - s\kappa^{-1}x_2 - c_2\kappa^{-3}\right).
\end{align*}
In the present validation, we applied $(c_1, c_2) = (c_{1L}, c_{2L})$ as well as the speed parameter $s$ are set as
\begin{align*}
&u_L = [1.46777062491],\quad v_L = [0.238709208571],\quad
s\in 0.44819467507505_{461}^{512},\\
&c_{1L}\in 1.25779442046144_{35}^{51},\quad c_{2L}\in -0.5207279753417_{6075}^{5985}\end{align*}
following the Rankine-Hugoniot relation (e.g., \cite{KK1990}),
where $[a]$ denotes the point interval consisting of a value $a$, and subscript and superscript numbers denote lower and upper bounds
of the interval, respectively.
We then compute an equilibrium on the horizon which satisfy
\[
	x_\ast\in\left(
	\begin{array}{c}
	\left[0.98913699589497727,0.98913699589497773\right]\\
	\left[0.20675855700518036,0.2067585570051809\right]
	\end{array}\right).
\]
Finally, validate blow-up solutions in the same way as the previous example.
Our validation result is listed in Table \ref{Tab:Ex2}.

\begin{table}[ht]
\caption{Validated results for (\ref{KK}): numerical validations prove $x(\tau_N)\in {\rm int}\,N$ and (\ref{tmax-estimate}) yields the inclusion of the blow-up time $t_{\max}$.}
\centering
\resizebox{\textwidth}{!}{
\begin{tabular}{ccccc}
\hline 
$(x_1(0),x_2(0))$ & $\epsilon$ & $\tau_N$ & $t_{\max}$ & Exec. time\\
\hline\\[-2mm]
$(-0.1,-0.8)$ & $0.00011049230488192128$ & $11.55312519434721$ & $0.944_{239514010626}^{69739415956034}$ & 0.86 s\\[1mm]
\hline 
\end{tabular}%
}
\label{Tab:Ex2}
\end{table}%

Validated results in this example show the efficiency of quasi-parabolic compactifications for asymptotically quasi-homogeneous vector fields at infinity.
As indicated in \cite{Mat}, quasi-Poincar\'{e} compactifications; namely, the case $\kappa(y) = (1+p(y)^{2c})^{1/2c}$, require calculations of {\em radicals} because of the presence of $\kappa^{-1}$ in desingularized vector fields.
Such terms cause the lack of smoothness of desingularized vector fields on the horizon, which indicates that the stability analysis of equilibria there in terms of Jacobian matrices makes no sense.
In particular, blow-up arguments cannot be developed within the present theory.
On the other hand, quasi-parabolic compactifications guarantees the smoothness of desingularized vector fields derived from original ones under their smoothness, {\em including the horizon}, by Proposition \ref{prop-para-extension}.
Blow-up arguments including numerical validations with quasi-parabolic compactifications can be therefore applied to vector fields which are not necessarily quasi-homogeneous.

%
%
\subsection{Example 3}
\label{section-ex3}
The final example is a finite dimensional approximation of the following system of partial differential equations:
\begin{equation}
\label{KS-radial}
\left\{\begin{array}{ll}
u_t=r^{1-d}\left(r^{d-1}\left(u_r-uv_r\right)\right)_r,&r\in(0,L),~t>0,\\
v_t=r^{1-d}\left(r^{d-1}v_r\right)_r-v+u&r\in(0,L),~t>0,\\
u_r=v_r=0,&r=0,L,~t>0,\\
u(r,0)=u_0(r),~v(r,0)=v_0(r),&r\in(0,L)
\end{array}
\right.
\end{equation}
for some $L > 0$, which is the well-known {\em Keller-Segel model} on the $d$-dimensional ball with homogeneous Neumann boundary condition and radially symmetric anzats:
\begin{equation}
\label{KS-PDE}
\left\{\begin{array}{ll}
u_t=\Delta u-\nabla\cdot(u\nabla v),&x\in \Omega,~t>0,\\
v_t=\Delta v-v+u,&x\in\Omega,~t>0,\\
\frac{\partial u}{\partial\nu}=\frac{\partial v}{\partial\nu}=0,&x\in\partial\Omega,~t>0,\\
u(x,0)=u_0(x),~v(x,0)=v_0(x),&x\in\Omega,
\end{array}
\right.
\end{equation}
where $\Omega = \{x\in \mathbb{R}^d \mid |x|<L\}$.

Zhou and Saito \cite{ZG2017} has proposed a {\em finite volume discretization scheme} so that blow-up solutions for (\ref{KS-radial}) of the parabolic-elliptic (namely, $v_t=0$) type can be computed\footnote{
It is known that solutions of the Keller-Segel system (\ref{KS-PDE}) with positive initial data $u_0(x) > 0$, $v_0(x) > 0$ must be positive.
Moreover, the system (\ref{KS-PDE}) possesses an $L^1$-conservation law for $u$; namely $\int_\Omega u(x,t)dx = \int_\Omega u_0(x)dx$ holds for all $t\geq 0$. 
However, $L^1$-conservative discretization schemes for (\ref{KS-PDE}) are known to possess {\em no} numerical blow-up solutions typically.
See \cite{ZG2017} for details.
}.
We consider a parabolic-parabolic alternative of the discretization defined below:
\begin{align*}
\frac{du_1}{dt}&=\frac{r_1^{1-d}}{h}\left(r_{1+\frac{1}{2}}^{d-1}\frac{u_2-u_1}{h}\right)-\frac{r_1^{1-d}}{h}\left(r_{1+\frac{1}{2}}^{d-1}\frac{v_2-v_1}{h}u_1\right)\\
\frac{du_2}{dt}&= \frac{r_2^{1-d}}{h}\left(r_{2+\frac{1}{2}}^{d-1}\frac{u_3-u_2}{h}-r_{2-\frac{1}{2}}^{d-1}\frac{u_2-u_1}{h}\right)-\frac{r_2^{1-d}}{h}\left(r_{2+\frac{1}{2}}^{d-1}\frac{v_3-v_2}{h}u_2-r_{2-\frac{1}{2}}^{d-1}\frac{v_2-v_1}{h}u_1\right)\\
&\vdots\\
\frac{du_i}{dt}&= \frac{r_i^{1-d}}{h}\left(r_{i+\frac{1}{2}}^{d-1}\frac{u_{i+1}-u_{i}}{h}-r_{i-\frac{1}{2}}^{d-1}\frac{u_{i}-u_{i-1}}{h}\right)-\frac{r_i^{1-d}}{h}\left(r_{i+\frac{1}{2}}^{d-1}\frac{v_{i+1}-v_{i}}{h}u_i-r_{i-\frac{1}{2}}^{d-1}\frac{v_{i}-v_{i-1}}{h}u_{i-1}\right)\\
&\vdots\\
\frac{du_{N}}{dt}&=\frac{r_{N}^{1-d}}{h}\left(-r_{N-\frac{1}{2}}^{d-1}\frac{u_{N}-u_{N-1}}{h}\right)-\frac{r_{N}^{1-d}}{h}\left(-r_{N-\frac{1}{2}}^{d-1}\frac{v_{N}-v_{N-1}}{h}u_{N}\right),
\end{align*}
and
\begin{align*}
\frac{dv_1}{dt}&=\frac{r_1^{1-d}}{h}\left(r_{1+\frac{1}{2}}^{d-1}\frac{v_2-v_1}{h}\right)-v_1+u_1\\
\frac{dv_2}{dt}&=\frac{r_2^{1-d}}{h}\left(r_{2+\frac{1}{2}}^{d-1}\frac{v_3-v_2}{h}-r_{2-\frac{1}{2}}^{d-1}\frac{v_2-v_1}{h}\right)-v_2+u_2\\
&\vdots\\
\frac{dv_i}{dt}&=\frac{r_i^{1-d}}{h}\left(r_{i+\frac{1}{2}}^{d-1}\frac{v_{i+1}-v_i}{h}-r_{i-\frac{1}{2}}^{d-1}\frac{v_i-v_{i-1}}{h}\right)-v_i+u_i\\
&\vdots\\
\frac{dv_{N}}{dt}&=\frac{r_{N}^{1-d}}{h}\left(-r_{N-\frac{1}{2}}^{d-1}\frac{v_{N}-v_{N-1}}{h}\right)-v_{N}+u_{N}.
\end{align*}
We name the system {\bf (FvKS)}.
The precise setting (FvKS) is as follows:
letting $N\in\mathbb{N}$ and $h=L/N$, the mesh of the interval $(0,L)\subset \mathbb{R}$ is defined by
\[
	0=r_{\frac{1}{2}}<r_{1+\frac{1}{2}}<\dots<r_{N-1+\frac{1}{2}}<r_{N+\frac{1}{2}}=L,
\]
where $r_{i+\frac{1}{2}}=ih$ ($i=0,1,\dots,N$).
In this example, we set $L=1$.
Here, $(r_{i+\frac{1}{2}},r_{i+1+\frac{1}{2}})$ ($i=0,1,\dots,N-1$) is called the control volume with its control point $r_{i+1}=(i+\frac{1}{2})h$.
The semi-discretization of the space variable yields the approximation satisfying $u_i(t)\simeq u(r_i,t)$ and $v_i(t)\simeq v(r_i,t)$ ($i=1,2,\dots,N$, $t>0$).

\begin{rem}\rm
We briefly gather several facts about blow-up behavior in the Keller-Segel systems of the parabolic-parabolic type (\ref{KS-PDE}).
In \cite{HV1997}, radially symmetric blow-up solutions for (\ref{KS-radial}) with $d\geq 2$ are constructed constitutively.
In \cite{OY2001}, the Keller-Segel system with $d=1$ is proved to admit {\em no} blow-up solutions.
In \cite{W2013}, criteria for blow-ups of radial-symmetric solutions for (\ref{KS-radial}) with $d\geq 3$ are provided.
In \cite{M2016}, radially symmetric blow-up solutions for (\ref{KS-radial}) with $d=2$ is proved to be of so-called {\em type II}; namely, asymptotics near blow-up is not determined only by nonlinearity of vector fields.
See references therein and others for more details.
\end{rem}

\par
\bigskip
First we observe that (FvKS) is asymptotically quasi-homogeneous in the following sense.
\begin{lem}
\label{lem-scaling-KS}
The system (FvKS) is an asymptotically quasi-homogeneous vector field at infinity of the following type and order $2$:
\begin{equation*}
\alpha = (\underbrace{2\dots,2}_{N}, \underbrace{1\dots,1}_{N}).
\end{equation*}
In other words, (FvKS) is asymptotically quasi-homogeneous under the scaling $u_i\mapsto s^2u_i$ and $v_i\mapsto sv_i$ for $i=1,\cdots, N$.
\end{lem}

Following Lemma \ref{lem-scaling-KS}, we consider two types of quasi-homogeneous compactifications.
One is the directional compactification of type $\alpha$:
\begin{equation}
\label{dir-KS}
u_1 = \frac{1}{s^2},\quad u_i = \frac{x_i}{s^2}\ (i=2,\cdots, N),\quad v_j = \frac{y_j}{s}\ (j=1,\cdots, N),
\end{equation}
and the other is the quasi-parabolic compactification of type $\alpha$:
\begin{equation}
\label{para-KS}
y_j = \frac{x_j}{(1-p(x)^{4})^{\alpha_j}}\ (j=1,\cdots, 2N),\quad p(x)^4 = \sum_{j=1}^{N}\bar u_j^2 + \bar v_j^4,\quad \kappa^{-1} = 1-p(x)^4,
\end{equation}
where $x = (x_1,\cdots, x_{2N}) \equiv (\bar u_1,\cdots, \bar u_N, \bar v_1, \cdots, \bar v_N)$.

%
%
\subsubsection{Directional compactification}
\label{section-Ex3-dir}

Direct computations yield the following transformation of vector fields:
\begin{align*}
u_1' &= -2s^{-3} s'\\
	&= \frac{r_1^{1-d}}{h}\left(r_{1+\frac{1}{2}}^{d-1}s^{-2}\frac{x_2-1}{h}\right)-\frac{r_1^{1-d}}{h}\left(r_{1+\frac{1}{2}}^{d-1}s^{-3}\frac{y_2-y_1}{h}\right),
\end{align*}
namely,
\begin{equation*}
s' = -\frac{r_1^{1-d}}{2h^2}r_{1+\frac{1}{2}}^{d-1} \left\{s(x_2-1) - (y_2-y_1)\right\}.
\end{equation*}
Similarly,
\begin{align*}
u_2' &= -2s^{-3}x_2 s' + s^{-2} x_2'\\
	&= \frac{r_2^{1-d}}{h} s^{-2} \left(r_{2+\frac{1}{2}}^{d-1}\frac{x_3-x_2}{h}-r_{2-\frac{1}{2}}^{d-1}\frac{x_2-1}{h}\right) - \frac{r_2^{1-d}}{h}s^{-3} \left(r_{2+\frac{1}{2}}^{d-1} \frac{y_3-y_2}{h}x_2 - r_{2-\frac{1}{2}}^{d-1}\frac{y_2-y_1}{h}\right),
\end{align*}
to obtain
\begin{align*}
x_2' &= 2s^{-1}x_2 s' + s^2 u_2'\\
	&= 2s^{-1}x_2 \left[ -\frac{r_1^{1-d}}{2h^2}r_{1+\frac{1}{2}}^{d-1} \left\{s(x_2-1) - (y_2-y_1)\right\}\right]\\
	&\quad + \frac{r_2^{1-d}}{h} \left(r_{2+\frac{1}{2}}^{d-1}\frac{x_3-x_2}{h}-r_{2-\frac{1}{2}}^{d-1}\frac{x_2-1}{h}\right) - \frac{r_2^{1-d}}{h}s^{-1} \left(r_{2+\frac{1}{2}}^{d-1} \frac{y_3-y_2}{h}x_2 - r_{2-\frac{1}{2}}^{d-1}\frac{y_2-y_1}{h}\right)\\
	&= -s^{-1}x_2 \left[ \frac{r_1^{1-d}}{h^2}r_{1+\frac{1}{2}}^{d-1} \left\{s(x_2-1) - (y_2-y_1)\right\}\right]\\
	&\quad + \frac{r_2^{1-d}}{h^2} \left[ r_{2+\frac{1}{2}}^{d-1} \left\{ (x_3-x_2) -s^{-1}(y_3-y_2) x_2 \right\} - r_{2-\frac{1}{2}}^{d-1} \left\{ (x_2-1) -s^{-1}(y_2-y_1)\right\} \right].
\end{align*}
For $u_i$ with $i= 3,\cdots, N=1$, 
\begin{align*}
u_i' &= -2s^{-3}x_i s' + s^{-2} x_i'\\
	&= \frac{r_i^{1-d}}{h}  s^{-2}  \left(r_{i+\frac{1}{2}}^{d-1}\frac{x_{i+1}-x_{i}}{h}-r_{i-\frac{1}{2}}^{d-1} \frac{x_{i} - x_{i-1}}{h}\right)-\frac{r_i^{1-d}}{h}  s^{-3} \left(r_{i+\frac{1}{2}}^{d-1}\frac{y_{i+1}-y_{i}}{h}x_i - r_{i-\frac{1}{2}}^{d-1}\frac{y_{i}-y_{i-1}}{h}x_{i-1}\right),
\end{align*}
to obtain
\begin{align*}
x_i' &= -2s^{-1}x_i s' + s^2 u_i'\\
	&= -s^{-1}x_i \left[ \frac{r_1^{1-d}}{h^2}r_{1+\frac{1}{2}}^{d-1} \left\{s(x_2-1) - (y_2-y_1)\right\}\right]\\
	&\quad + \frac{r_i^{1-d}}{h^2} \left[ r_{i+\frac{1}{2}}^{d-1} \left\{ (x_{i+1}-x_{i}) - s^{-1}  (y_{i+1}-y_{i}) x_i \right\} - r_{i-\frac{1}{2}}^{d-1} \left\{ (x_{i} - x_{i-1}) - s^{-1}(y_{i}-y_{i-1})x_{i-1}\right\} \right].
\end{align*}
Finally, 
\begin{align*}
u_N' &= -2s^{-3}x_N s' + s^{-2} x_N'\\
	&=\frac{r_{N}^{1-d}}{h} s^{-2}\left(-r_{N-\frac{1}{2}}^{d-1}\frac{x_{N}-x_{N-1}}{h}\right)-\frac{r_{N}^{1-d}}{h} s^{-2} \left(-r_{N-\frac{1}{2}}^{d-1}\frac{y_{N}-y_{N-1}}{h}x_{N}\right)\end{align*}
to obtain
\begin{align*}
x_N' &= -2s^{-1}x_N s' + s^2 u_N'\\
	&= -s^{-1}x_N \left[ \frac{r_1^{1-d}}{h^2}r_{1+\frac{1}{2}}^{d-1} \left\{s(x_2-1) - (y_2-y_1)\right\}\right]- \frac{r_{N}^{1-d}}{h^2} r_{N-\frac{1}{2}}^{d-1} \left\{ (x_{N}-x_{N-1}) - s^{-1}(y_{N}-y_{N-1})x_{N}\right\}.
\end{align*}
Next compute $y_i'$.
\begin{align*}
v_1'&= -s^{-2} y_1 s' + s^{-1} y_1'
	= \frac{r_1^{1-d}}{h}s^{-1}\left(r_{1+\frac{1}{2}}^{d-1}\frac{y_2-y_1}{h}\right)- s^{-1}y_1+ s^{-2}
\end{align*}
to obtain
\begin{align*}
y_1'&= s^{-1} y_1 s' + s v_1'\\
	&= -s^{-1}y_1 \left[ \frac{r_1^{1-d}}{2h^2}r_{1+\frac{1}{2}}^{d-1} \left\{s(x_2-1) - (y_2-y_1)\right\}\right]
	+ \frac{r_1^{1-d}}{h^2}r_{1+\frac{1}{2}}^{d-1} (y_2-y_1)- y_1+ s^{-1}.
\end{align*}
Similarly,
\begin{align*}
v_i'&= -s^{-2} y_i s' + s^{-1} y_i'\\
	&= \frac{r_i^{1-d}}{h^2}s^{-1}\left(r_{i+\frac{1}{2}}^{d-1}(y_{i+1}-y_i) -r_{i-\frac{1}{2}}^{d-1} (y_i-y_{i-1}) \right) - s^{-1}y_i + s^{-2}x_i
\end{align*}
to obtain
\begin{align*}
y_i'&= s^{-1} y_i s' + s v_i'\\
	&= -s^{-1}y_i \left[ \frac{r_1^{1-d}}{2h^2}r_{1+\frac{1}{2}}^{d-1} \left\{s(x_2-1) - (y_2-y_1)\right\}\right]\\
	&+ \frac{r_i^{1-d}}{h^2} \left(r_{i+\frac{1}{2}}^{d-1}(y_{i+1}-y_i) -r_{i-\frac{1}{2}}^{d-1} (y_i-y_{i-1}) \right) - y_i + s^{-1}x_i,\quad i=2,\cdots, N-1,
\end{align*}
and
\begin{align*}
v_N'&= -s^{-2} y_N s' + s^{-1} y_N'\\
	&= \frac{r_{N}^{1-d}}{h^2}s^{-1}\left(-r_{N-\frac{1}{2}}^{d-1}(y_{N}-y_{N-1}) \right)- s^{-1}y_{N}+ s^{-2}x_{N}
\end{align*}
to obtain
\begin{align*}
y_N'&= s^{-1} y_N s' + s v_N'\\
	&= -s^{-1}y_N \left[ \frac{r_1^{1-d}}{2h^2}r_{1+\frac{1}{2}}^{d-1} \left\{s(x_2-1) - (y_2-y_1)\right\}\right]
	+ \frac{r_{N}^{1-d}}{h^2} r_{N-\frac{1}{2}}^{d-1}\left(-(y_{N}-y_{N-1}) \right)- y_{N}+ s^{-1}x_{N}.
\end{align*}
Introducing the time-variable desingularization
\begin{equation*}
\frac{d\tau}{dt} = s^{-1},
\end{equation*}
we have the following result.

\begin{lem}
\label{lem-dir-KS}
The desingularized vector field of (FvKS) with respect to the directional compactification (\ref{dir-KS}) is the following system:
\begin{align*}
\dot s &= - s\frac{r_1^{1-d}}{2h^2}r_{1+\frac{1}{2}}^{d-1} \left\{s(x_2-1) - (y_2-y_1)\right\},\\
\dot x_2 &= -x_2 \left[ \frac{r_1^{1-d}}{h^2}r_{1+\frac{1}{2}}^{d-1} \left\{s(x_2-1) - (y_2-y_1)\right\}\right]\\
	&\quad + \frac{r_2^{1-d}}{h^2} \left[ r_{2+\frac{1}{2}}^{d-1} \left\{ s (x_3-x_2) - (y_3-y_2) x_2 \right\} - r_{2-\frac{1}{2}}^{d-1} \left\{ s(x_2-1) -(y_2-y_1)\right\} \right],\\
%
%
\dot x_i &= - x_i \left[ \frac{r_1^{1-d}}{h^2}r_{1+\frac{1}{2}}^{d-1} \left\{s(x_2-1) - (y_2-y_1)\right\}\right]\\
	&\quad + \frac{r_i^{1-d}}{h^2} \left[ r_{i+\frac{1}{2}}^{d-1} \left\{ s(x_{i+1}-x_{i}) - (y_{i+1}-y_{i}) x_i \right\} - r_{i-\frac{1}{2}}^{d-1} \left\{ s(x_{i} - x_{i-1}) - (y_{i}-y_{i-1})x_{i-1}\right\} \right],\\
	&\quad\quad  (i=3,\cdots, N-1)\\
%
%
\dot x_N &= - x_N \left[ \frac{r_1^{1-d}}{h^2}r_{1+\frac{1}{2}}^{d-1} \left\{s(x_2-1) - (y_2-y_1)\right\}\right]- \frac{r_{N}^{1-d}}{h^2} r_{N-\frac{1}{2}}^{d-1} \left\{ s (x_{N}-x_{N-1}) - (y_{N}-y_{N-1})x_{N-1}\right],
\end{align*}
\begin{align*}
\dot y_1 &= - y_1 \left[ \frac{r_1^{1-d}}{2h^2}r_{1+\frac{1}{2}}^{d-1} \left\{s(x_2-1) - (y_2-y_1)\right\}\right]
	+ s \frac{r_1^{1-d}}{h^2}r_{1+\frac{1}{2}}^{d-1} (y_2-y_1) - s y_1+ 1,\\
%
%
\dot y_i &= -y_i \left[ \frac{r_1^{1-d}}{2h^2}r_{1+\frac{1}{2}}^{d-1} \left\{s(x_2-1) - (y_2-y_1)\right\}\right]\\
	&+ s \frac{r_i^{1-d}}{h^2} \left(r_{i+\frac{1}{2}}^{d-1}(y_{i+1}-y_i) -r_{i-\frac{1}{2}}^{d-1} (y_i-y_{i-1}) \right) - s y_i + x_i,\quad (i=2,\cdots, N-1)\\
%
%
\dot y_N &= - y_N \left[ \frac{r_1^{1-d}}{2h^2}r_{1+\frac{1}{2}}^{d-1} \left\{s(x_2-1) - (y_2-y_1)\right\}\right]
	+ s \frac{r_{N}^{1-d}}{h^2} r_{N-\frac{1}{2}}^{d-1}\left(-(y_{N}-y_{N-1}) \right)- s y_{N}+ x_{N}.
\end{align*}
\end{lem}

Our concerning blow-up solution is a trajectory of the desingularized vector field asymptotic to an equilibrium on the horizon $\{s=0\}$.
The initial data is given by
\begin{equation}
\label{initial-ex3}
u_i(0)=100(1+\cos(\pi r_i)),~v_i(0)=0~(i=1,2,\dots,N).
\end{equation}
Then, we derive
\[
	s(0)=\frac{1}{\sqrt{u_1(0)}},~x_i(0)=\frac{u_i(0)}{u_1(0)}~(i=2,3,\dots,N),~y_j(0)=\frac{v_j(0)}{\sqrt{u_1(0)}}~(j=1,2,\dots,N).
\]	

Following Algorithm \ref{alg-validation2}, we validate global trajectories for the vector field in Lemma \ref{lem-dir-KS} asymptotic to $\mathcal{E} = \{s=0\}$ with various $(d,N)$.
Validated equilibria are near
\begin{align*}
&s = 0,\quad x_1 = -0.036653902557231,\quad x_2 = -8.275562067652\times 10^{-5},\quad x_j = 0\ (j\geq 3),\\
&y_1 = 0.04910809766161,\quad y_2 = 0.001800003426459655,\quad y_j = 0\ (j\geq 3),\quad \text{etc.}
\end{align*}
Validated results are collected in Table \ref{Tab:Ex3_1}, which correspond to rigorous enclosures of a trajectory illustrated in Figure \ref{fig-Ex3}.

\begin{table}[ht]
\caption{Validated results for (FvKS) using the directional compactification: numerical validations prove $x(\tau_N)\in {\rm int}\,N$ and (\ref{tmax-estimate-dir}) yields the inclusion of the blow-up time $t_{\max}$. Subscript and superscript numbers in the table denote lower and upper bounds of the interval, respectively.}
\centering
\resizebox{\textwidth}{!}{
\begin{tabular}{ccccc}
\hline 
$(d,N)$ & $\epsilon$ & $\tau_N$ & $t_{\max}$ & Exec. time\\
\hline\\[-2mm]
$(4,4)$ & $7.7787964060071189\times 10^{-7}$ & $2.2660030304331925$ & $
0.04163_{4995298971515}^{5093439395401}$ & 4.58 s\\[1mm]
$(4,5)$ & $3.9917525258063959\times 10^{-7}$ & $2.0798564005033283$ & $0.0414601_{47411111418}^{84989963539}$ & 11.33 s\\[1mm]
$(4,6)$ & $2.2532402360440276\times 10^{-7}$ & $1.9152197502002851$ & $0.0406815_{24736414453}^{44554173229}$ & 26.17 s\\[1mm]
$(4,7)$ & $4.5949729863572216\times 10^{-10}$ & $2.1715368022411817$ & $0.039930492_{160482736}^{204807836}$ & 57.81 s\\[1mm]
$(4,8)$ & $1.0\times 10^{-10}$ & $2.0850477418274505$ & $0.0394018633_{68052715}^{77681091}$ & 1 m 50.67 s\\[1mm]
$(4,9)$ & $1.1000000000000001\times 10^{-10}$ & $1.9431916522110496$ & $0.0390410168_{12640577}^{21151034}$ & 2 m 56.85 s\\[1mm]
$(4,10)$ & $1.4641000000000004\times 10^{-10}$ & $1.8118057787224227$ & $0.0387870664_{76778042}^{87772283}$ & 4 m 19.01 s\\[1mm]
$(4,11)$ & $1.1000000000000001\times 10^{-10}$ & $1.7422008610746525$ & $0.0386048800_{46284047}^{52778686}$ & 6 m 21.54 s\\[1mm]
$(4,12)$ & Failed & - & - & -\\[1mm]
$(3,4)$ & $1.2527829399838528\times 10^{-6}$ & $2.8164262707985448$ & $0.044016_{34379731982}^{564692126309}$ & 3.88 s\\[1mm]
$(3,5)$ & $5.8443248730331463\times 10^{-7}$ & $2.4889023211163482$ & $0.042811_{321959989476}^{448911760066}$ & 9.41 s\\[1mm]
$(3,6)$ & $3.6288659325512687\times 10^{-7}$ & $2.2711479488006821$ & $0.042214_{039058911502}^{116450095999}$ & 20.76 s\\[1mm]
$(3,7)$ & $2.2532402360440276\times 10^{-7}$ & $2.1201089281490533$ & $0.0417731_{52715466201}^{99179381021}$& 39.47 s\\[1mm]
$(3,8)$ & $3.684227838451178\times 10^{-8}$ & $2.1300729396508551$ & $0.0414018_{28395814903}^{34983204045}$ & 1 m 22.44 s\\[1mm]
$(3,9)$ & $1.3310000000000004\times 10^{-10}$ & $2.4057492533283767$ & $0.0411075733_{70292957}^{9831362}$ & 2 m 27.18 s\\[1mm]
$(3,10)$ & $5.5599173134922393\times 10^{-10}$ & $2.1386518033144264$ & $0.040888914_{45496233}^{652944385}$ & 3 m 16.68 s\\[1mm]
$(3,11)$ & $1.4641000000000004\times 10^{-10}$ & $2.102567451906797$ & $0.040731730_{763463577}^{868847683}$ & 5 m 9.42 s\\[1mm]
$(3,12)$ & Failed & - & - & -\\[1mm]
$(2,4)$ & $1.8341995024303595\times 10^{-7}$ & $3.5400444623271277$ & $0.05263_{7126736797233}^{9096803538601}$ & 3.65 s\\[1mm]
$(2,5)$ & Failed & - & - & -\\[1mm]
\hline 
\end{tabular}%
}
\label{Tab:Ex3_1}
\end{table}

\begin{figure}[htbp]\em
\begin{minipage}{0.5\hsize}
\centering
\includegraphics[width=7.0cm]{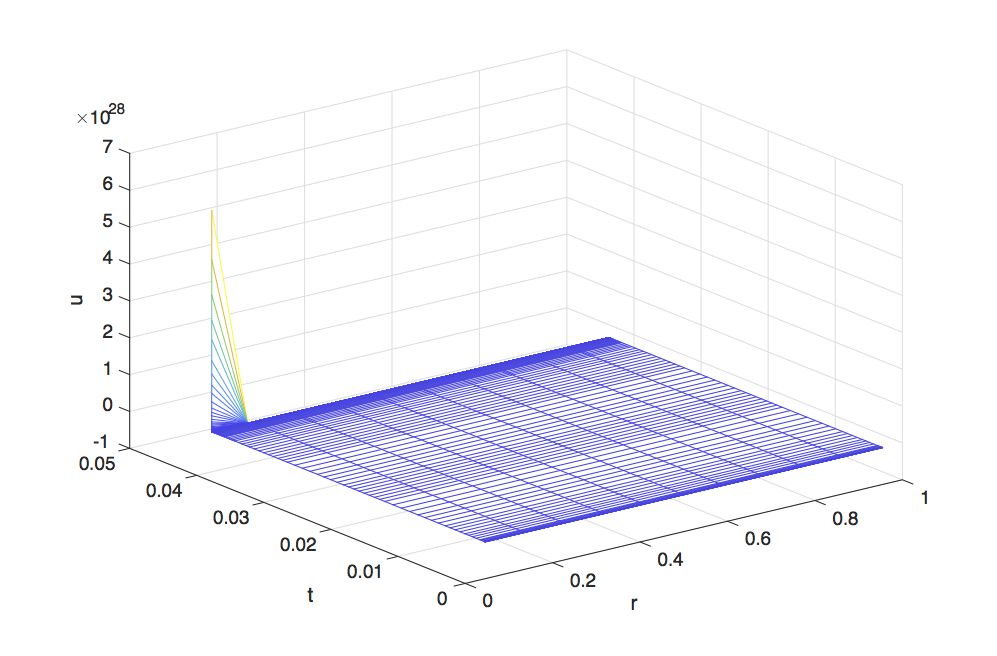}
(a)
\end{minipage}
\begin{minipage}{0.5\hsize}
\centering
\includegraphics[width=7.0cm]{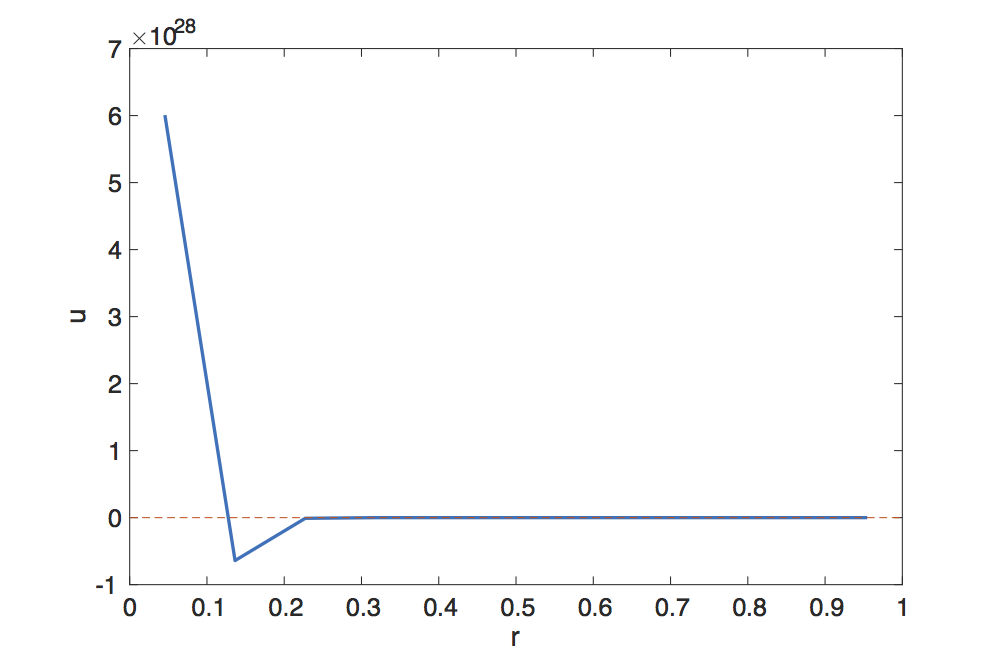}
(b)
\end{minipage}\\
\begin{minipage}{0.5\hsize}
\centering
\includegraphics[width=7.0cm]{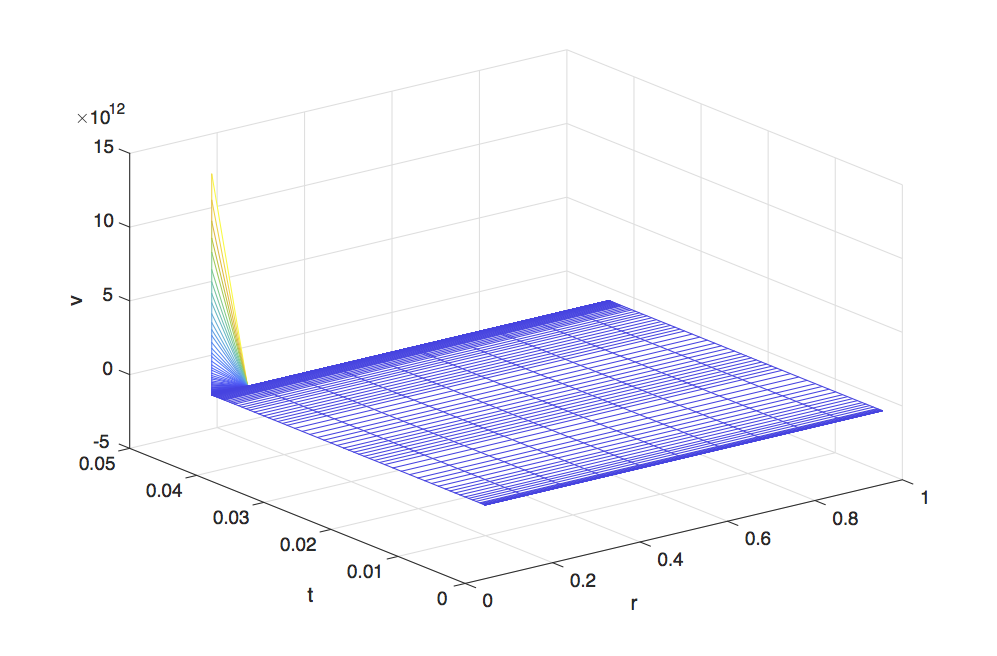}
(c)
\end{minipage}
\begin{minipage}{0.5\hsize}
\centering
\includegraphics[width=7.0cm]{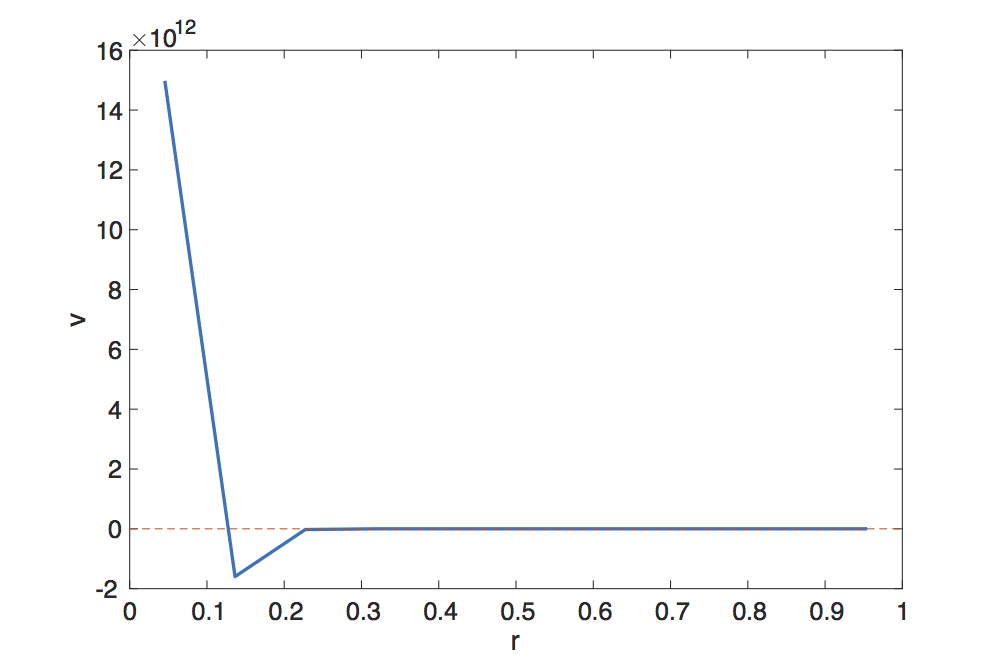}
(d)
\end{minipage}\\
\caption{A blow-up trajectory for (FvKS) with $(d,N)=(3,11)$}
\label{fig-Ex3}
A blow-up trajectory with the initial data (\ref{initial-ex3}) are drawn.
(a) : the $(t,r,u)$-plot of the blow-up trajectory. 
(b) : the $(r,u)$-plot of the blow-up trajectory near $t=t_{\max} \approx 0.04$. 
(a) : the $(t,r,v)$-plot of the blow-up trajectory. 
(b) : the $(r,v)$-plot of the blow-up trajectory near $t=t_{\max} \approx 0.04$. 
\end{figure}

\begin{rem}
The statement \lq\lq Failed" comes from the failure of Step 2 in Algorithm \ref{alg-validation2}.
That is, the matrix $A(x)$ could not be validated to be negative definite, although corresponding equilibria admit only eigenvalues with negative real parts (at least in the numerical sense).
This might be caused by the change of eigenvalue distributions of matrices $Df(x)$ via their symmetrizations.
\par
On the other hand, several eigenvalues of $Df(x_\ast)$ at equilibria $x_\ast$ are actually accumulated in the numerical sense, which implies that the computed eigenvectors may be linearly dependent. 
In such a case, we cannot apply the eigenmatrix diagonalizing $Df(x_\ast)$ to determining the matrix $Y$ in Proposition \ref{prop-Lyap}.
Instead we apply the Schur decomposition of $Df(x_\ast)$ to checking eigenvalues, and to determining $Y$; $Y=I$, as indicated in Remark \ref{rem-Lyap}.
\par
The similar cases occur for the system (FvKS) with quasi-parabolic compactifications.
\end{rem}

%
%
\subsubsection{Quasi-parabolic compactification}
Let
\begin{equation*}
\tilde f_j(x_1,\cdots, x_{2N}) := \kappa^{-(1+\alpha_j)}\tilde f_j(\kappa^2 x_1,\cdots, \kappa^2 x_N, \kappa x_{N+1}, \cdots, \kappa x_{2N}).
\end{equation*}
Then we have
\begin{align*}
\tilde f_1 &=\frac{r_1^{1-d}}{h}\left(r_{1+\frac{1}{2}}^{d-1}\frac{\bar u_2 - \bar u_1}{h}\right)\kappa^{-1} - \frac{r_1^{1-d}}{h}\left(r_{1+\frac{1}{2}}^{d-1}\frac{\bar v_2 - \bar v_1}{h}\bar u_1\right),\\
\tilde f_j &= \frac{r_j^{1-d}}{h}\left(r_{j+\frac{1}{2}}^{d-1}\frac{\bar u_{j+1} - \bar u_{j}}{h}-r_{j-\frac{1}{2}}^{d-1}\frac{\bar u_{j} - \bar u_{j-1}}{h}\right) \kappa^{-1} - \frac{r_j^{1-d}}{h}\left(r_{j+\frac{1}{2}}^{d-1}\frac{\bar v_{j+1} - \bar v_{j}}{h}\bar u_j - r_{j-\frac{1}{2}}^{d-1}\frac{\bar v_{j} - \bar v_{j-1}}{h}\bar u_{j-1}\right),\\
&\quad\quad  (j=2,\cdots, N-1)\\
\tilde f_N &=\frac{r_{N}^{1-d}}{h}\left(-r_{N-\frac{1}{2}}^{d-1}\frac{\bar u_{N} - \bar u_{N-1}}{h}\right)\kappa^{-1} -\frac{r_{N}^{1-d}}{h}\left(-r_{N-\frac{1}{2}}^{d-1}\frac{\bar v_{N} - \bar v_{N-1}}{h}\bar u_{N}\right),
\end{align*}
and
\begin{align*}
\tilde f_{N+1} &=\frac{r_1^{1-d}}{h}\left(r_{1+\frac{1}{2}}^{d-1}\frac{\bar v_2 - \bar v_1}{h}\right) \kappa^{-1} - \bar v_1 \kappa^{-1}+ \bar u_1,\\
\tilde f_{N+j} &=\frac{r_j^{1-d}}{h}\left(r_{j+\frac{1}{2}}^{d-1}\frac{\bar v_{j+1}-\bar v_j}{h}-r_{j-\frac{1}{2}}^{d-1}\frac{\bar v_j - \bar v_{j-1}}{h}\right) \kappa^{-1} - \bar v_j \kappa^{-1} + \bar u_j,\\
&\quad\quad  (j=2,\cdots, N-1)\\
\tilde f_{2N} &=\frac{r_{N}^{1-d}}{h}\left(-r_{N-\frac{1}{2}}^{d-1}\frac{\bar v_{N} - \bar v_{N-1}}{h}\right) \kappa^{-1} - \bar v_{N} \kappa^{-1} + \bar u_{N}.
\end{align*}

Recall that the desingularized vector field associated with the vector field $y'=f(y)$ on $\mathbb{R}^{2N}$ with quasi-parabolic compactification of type $\alpha$ is (\ref{ODE-desing-para}).
The sum $G(x) := \sum_{j=1}^{2N} \frac{x_j^{2\beta_j-1}}{\alpha_j} \tilde f_j(x)$ is necessary to be computed.
Now we have
\begin{align*}
\sum_{j=1}^{N} \frac{x_j^{2\beta_j-1}}{\alpha_j} \tilde f_j(x) &= \sum_{j=1}^{N} \frac{\bar u_j}{2} \tilde f_j(x)\\
	&= \frac{1}{2h^2}\sum_{j=1}^{N-1} r_{j+\frac{1}{2}}^{d-1} \{-r_{j+1}^{1-d}\bar u_{j+1} + r_j^{1-d}\bar u_j \}\cdot  \{\kappa^{-1}(\bar u_{j+1} - \bar u_j) - (\bar v_{j+1}-\bar v_j)\bar u_j \},\\
\sum_{j=N+1}^{2N} \frac{x_j^{2\beta_j-1}}{\alpha_j} \tilde f_j(x) &= \sum_{j=1}^{N} \bar v_j^3 \tilde f_{N+j}(x)\\
	&= \frac{\kappa^{-1}}{h^2}	\sum_{j=1}^{N-1} r_{j+\frac{1}{2}}^{d-1} \{-r_{j+1}^{1-d}\bar v_{j+1}^3+ r_j^{1-d}\bar v_j^3 \} (\bar v_{j+1} - \bar v_j) - \kappa^{-1}\sum_{j=1}^N \bar v_j^4 +\sum_{j=1}^N \bar v_j^3 \bar u_j.
\end{align*}
Therefore we have
\begin{align}
\notag
G(x)	&= \frac{1}{2h^2}\sum_{j=1}^{N-1} r_{j+\frac{1}{2}}^{d-1} \{-r_{j+1}^{1-d}\bar u_{j+1} + r_j^{1-d}\bar u_j \}\cdot  \{\kappa^{-1}(\bar u_{j+1} - \bar u_j) - (\bar v_{j+1}-\bar v_j)\bar u_j \}\\
\label{Gx}
	&+ \frac{\kappa^{-1}}{h^2}	\sum_{j=1}^{N-1} r_{j+\frac{1}{2}}^{d-1} \{-r_{j+1}^{1-d}\bar v_{j+1}^3+ r_j^{1-d}\bar v_j^3 \} (\bar v_{j+1} - \bar v_j) - \kappa^{-1}\sum_{j=1}^N \bar v_j^4 +\sum_{j=1}^N \bar v_j^3 \bar u_j.
\end{align}

Summarizing these arguments, we have the concrete form of the desingularized vector field:

\begin{lem}
\label{lem-para-KS}
The desingularized vector field for (FvKS) with the quasi-parabolic compactification of type $\alpha$ is the following:
\begin{align*}
\frac{d\bar u_i}{d\tau} &= \frac{1}{4}\left(1+3\sum_{j=1}^N (\bar u_j^2 + \bar v_j^4)\right) \tilde f_i(x) - 2\bar u_i G(x),\quad i=1,\cdots, N,\\
\frac{d\bar v_i}{d\tau} &= \frac{1}{4}\left(1+3\sum_{j=1}^N (\bar u_j^2 + \bar v_j^4)\right) \tilde f_{N+i}(x) - \bar v_i G(x),\quad i=1,\cdots, N,
\end{align*}
where $x=(x_1,\cdots, x_{2N})\equiv (\bar u_1,\cdots, \bar u_N, \bar v_1,\cdots, \bar v_N)$ and $G(x)$ is given in (\ref{Gx}).
\end{lem}

Our concerning blow-up solution is a trajectory of the desingularized vector field asymptotic to an equilibrium on the horizon $\{p(x) = 1\}$, which generally depends on $(d,N)$, while it corresponds to a point validated in Section \ref{section-Ex3-dir}.
Following Algorithm \ref{alg-validation1}, we validate global trajectories for the vector field in Lemma \ref{lem-para-KS} asymptotic to $\mathcal{E} = \partial \mathcal{D}$.
The initial data are set as (\ref{initial-ex3}) with application of $T_{para}$.
As for computations of $\kappa(y)$, we have applied the Krawczyk method (e.g., \cite{T2011}) to $F_y(\kappa)= \kappa^4 - \kappa^3 - p(y)^4 = 0$ appeared in Lemma \ref{lem-zero-compactification}.
Final validated results are collected in Table \ref{Tab:Ex3_2}.

\begin{table}[ht]
\caption{Validated results for (FvKS) using the quasi-parabolic compactification: numerical validations prove $x(\tau_N)\in {\rm int}\,N$ and (\ref{tmax-estimate}) yields the inclusion of the blow-up time $t_{\max}$.}
\centering
\resizebox{\textwidth}{!}{
\begin{tabular}{ccccc}
\hline 
$(d,N)$ & $\epsilon$ & $\tau_N$ & $t_{\max}$ & Exec. time\\
\hline\\[-2mm]
$(4,4)$ & $1.5389933993880384 \times 10^{-7}$ & $2.5104000513035319$ & $0.041635_{002136609429}^{154750508511}$ & 1 m 52.82 s\\[1mm]
$(4,5)$ & $9.5559381772732721\times 10^{-8}$ & $2.3018259253322216$ & $
0.041460_{149021090166}^{225199701329}$ & 4 m 00.40 s\\[1mm]
$(4,6)$ & $5.9334857761040084\times 10^{-8}$ & $2.151636139653439$ & $
0.0406815_{25425984681}^{65088442458}$ & 8 m 32.33 s\\[1mm]
$(4,7)$ & $1.1739085287969579\times 10^{-8}$ & $2.1158872071025688$ & $
0.03993049_{2091754095}^{9115842435}$ & 15 m 39.75 s\\[1mm]
$(4,8)$ & $1.6105100000000006\times 10^{-10}$ & $2.2551921883785618$ & $0.039401863_{359373755}^{463992299}$ & 30 m 21.43 s\\[1mm]
$(4,9)$ & $1.4641000000000004\times 10^{-10}$ & $2.1389222924223637$ & $
0.0390410168_{01699031}^{99000650}$ & 52 m 56.17 s\\[1mm]
$(4,10)$ & Failed & - & - & - \\[1mm]
$(3,4)$ & $2.2532402360440276\times 10^{-7}$ & $3.1345969238600971$ & $
0.044016_{358467806576}^{898408608799}$ & 1 m 21.04 s\\[1mm]
$(3,5)$ & $1.2718953713950728\times 10^{-7}$ & $2.8105206084304078$ & $
0.042811_{328397618989}^{795404836206}$ & 2 m 49.74 s\\[1mm]
$(3,6)$ & Failed & - & - & - \\[1mm]
$(2,4)$ & Failed & - & - & - \\[1mm]
\hline 
\end{tabular}%
}
\label{Tab:Ex3_2}
\end{table}

\subsubsection{Final remark: Scalings for (FvKS)}
The scaling derived in Lemma \ref{lem-scaling-KS} does {\em not} actually reflect the scaling in the original Keller-Segel system (\ref{KS-PDE}).
Indeed, the system (\ref{KS-PDE}) replacing the second equation by $v_t = \Delta v + u$ possesses the following scaling invariance:
\begin{equation}
\label{scaling-KS-PDE}
u_\lambda(x,t) := \lambda^2 u(\lambda x, \lambda^2 t),\quad v_\lambda(x,t) := v(\lambda x, \lambda^2 t),\quad \lambda > 0.
\end{equation}
In particular, the value of $v$ is not scaled, which is different from the type derived in Lemma \ref{lem-scaling-KS}.
We can consider another scaling to (FvKS) regarding the grid size parameter $h$ as an independent variable. 
Actually, we have the following scaling law for (FvKS), which will reflect the scaling (\ref{scaling-KS-PDE}).
\begin{lem}
Regard $h$ as an independent variable with trivial time evolution $dh/dt = 0$.
Then the system (FvKS) is asymptotically quasi-homogeneous of the following type and order $3$:
\begin{equation*}
\alpha = (\underbrace{2\dots,2}_{N}, \underbrace{0\dots,0}_{N}, -1)
\end{equation*}
with natural extension of type for nonpositive integers.
In other words, (FvKS) is asymptotically quasi-homogeneous under the scaling $u_i\mapsto s^2u_i$, $v_i\mapsto v_i$ for $i=1,\cdots, N$ and $h\mapsto s^{-1}h$.
\end{lem}
The authors have tried computing trajectories asymptotic to the horizon (for directional compactifications) with the above scaling, but they could find {\em no} such trajectories.
The scaling $h\mapsto s^{-1}h$ has a potential to link a rescaling algorithm for numerics of partial differential equations (e.g. \cite{BK1988}).

\section*{Conclusion}
In the present paper, we have derived a numerical validation procedure of blow-up solutions for vector fields with asymptotic quasi-homogeneity at infinity.
Our proposing numerical validation methodology is essentially the same as the previous study by authors and their collaborators \cite{TMSTMO} except the mathematical formulation of compactifications as well as time-variable desingularizations.
We have applied quasi-homogeneous compactifications to describing the infinity so that the desingularized vector field for asymptotically quasi-homogeneous ones can appropriately describe dynamics at infinity.
\par
In the present paper, we have also introduced a new quasi-homogeneous compactification called quasi-parabolic one, which is an alternative of the quasi-Poincar\'{e} compactification \cite{Mat}.
This compactification determines a global chart unlike directional compactifications, and overcomes the lack of smoothness of desingularized vector fields at infinity which arise in cases of Poincar\'{e}-type compactifications.
The former property enables us to validate blow-up solutions through sign-changing trajectories (Section \ref{section-ex1}), and the latter enables us to apply our validation procedure to asymptotically quasi-homogeneous vector fields (Sections \ref{section-ex2} and \ref{section-ex3}).
Quasi-homogeneous compactifications such as directional and admissible quasi-homogeneous ones will open the door to numerical validations of blow-up solutions for various polynomial vector fields including finite dimensional approximations of {\em systems of partial differential equations}.

\section*{Acknowledgements}
KM was partially supported by Program for Promoting the reform of national universities (Kyushu University), Ministry of Education, Culture, Sports, Science and Technology (MEXT), Japan, World Premier International Research Center Initiative (WPI), MEXT, Japan, and JSPS Grant-in-Aid for Young Scientists (B) (No. 17K14235).
AT was partially supported by JSPS Grant-in-Aid for Young Scientists (B) (No. 15K17596).

\appendix
\section{Schur decompositions}
\label{appendix-Schur}

In this section we review Schur decompositions of squared matrices.
\begin{prop}[Schur decomposition, e.g., \cite{GL1996}]
\label{prop-Schur1}
Let $A\in M_n(\mathbb{C})$ : complex $n\times n$ matrix.
Then there exists a unitary matrix $Q\in U(n)$ such that
\begin{equation*}
Q^H A Q = T \equiv D + N,
\end{equation*}
where $Q^H$ is the Hermitian transpose of $Q$, $D = \diag(\lambda_1,\cdots, \lambda_n)$ and $N\in M_n(\mathbb{C})$ is strictly upper triangular.
Furthermore, $Q$ can be chosen so that the eigenvalues $\lambda_i$ appear in any order along the diagonal.
We shall call $T$ {\em a Schur normal form of $A$}.
\end{prop}

When we treat all computations in {\em real} floating number or interval arithmetic, the real version of Schur decompositions can be applied.

\begin{prop}[Real Schur decomposition, e.g., \cite{GL1996}]
\label{prop-Schur2}
Let $A\in M_n(\mathbb{R})$ : real $n\times n$ matrix.
Then there exists an orthogonal matrix $Q\in O(n)$ such that
\begin{equation*}
Q^T A Q = T\equiv \begin{pmatrix}
R_{11} & R_{12} & \cdots & R_{1m} \\
0 & R_{22} & \cdots & R_{2m} \\
\vdots & \vdots & \ddots & \vdots \\
0 & 0 & \cdots & R_{mm} 
\end{pmatrix},
\end{equation*}
where each $R_{ii}$ is either a $1\times 1$ or a $2\times 2$ matrix having complex conjugate eigenvalues.
We shall call $T$ {\em a real Schur normal form of $A$}.
\end{prop}

A merit of Schur decompositions is that we can apply it to {\em arbitrary} square matrices.
In particular, change of coordinates via Schur decompositions can be realized {\em no matter what the multiplicities of any eigenvalues are}.

\section{Concrete calculations of an upper bound of $t_{\max}$ with quasi-parabolic compactifications}
\label{appendix-estimates}

In this section, we consider the rigorous validation of the maximal existence time
\[
	t_{\max} = \int_0^\infty \kappa^{-k}\left(1-\frac{2c-1}{2c}\kappa^{-1}\right)d\tau
\]
of solution trajectories with quasi-parabolic compactifications and computer assistance.
\par
First of all, we compute the following integral representing the time of integration of computed trajectory for desingularized vector fields {\em in $t$-timescale} in advance:
\[
	t_N=\int_0^{\tau_N}\left(1-p(x(\tau))^{2c}\right)^{k}\left(1-\frac{2c-1}{2c}\left(1-p(x(\tau))^{2c}\right)\right)d\tau.
\]
As mentioned in Section \ref{section-explicit}, the estimate of $|1-p(x)^{2c}|$ is essential to computation of an upper bound $C_{n,\alpha,N}(L)$.
At first, we derive the estimate with the type $\alpha=(1,2)$ and $c=2$ as an example.
Let $x_\ast = (x^\ast_1, x^\ast_2) \in \mathcal{E}$ and assume that a Lyapunov function $L(x)$ is validated in a vicinity of $x_\ast$. 
Then
\begin{align*}
x_1^4+x_2^2
=&(x_1-x_1^\ast+x_1^\ast)^4+(x_2-x_2^\ast+x_2^\ast)^2\\
=&(x_1-x_1^\ast)^4+4(x_1-x_1^\ast)^3x_1^\ast+6(x_1-x_1^\ast)^2(x_1^\ast)^2+4(x_1-x_1^\ast)(x_1^\ast)^3+(x_1^\ast)^4\\
&+(x_2-x_2^\ast)^2+2(x_2-x_2^\ast)x_2^\ast+(x_2^\ast)^2.
\end{align*}
Now $p(x_\ast)=1$ holds since $x_\ast\in \mathcal{E}$.
Thus we have
\begin{align*}
\left|1-p(x)^{2c}\right|
=&\Big|(x_1-x_1^\ast)^4+4(x_1-x_1^\ast)^3x_1^\ast+6(x_1-x_1^\ast)^2(x_1^\ast)^2+4(x_1-x_1^\ast)(x_1^\ast)^3\\
&+(x_2-x_2^\ast)^2+2(x_2-x_2^\ast)x_2^\ast\Big|\\
=&\left|\left[4(x_1^\ast)^3~2x_2^\ast\right]\left[\begin{array}{l}
x_1-x_1^\ast\\x_2-x_2^\ast
\end{array}\right]+\left[6(x_1^\ast)^2~1\right]\left[\begin{array}{l}
(x_1-x_1^\ast)^2\\(x_2-x_2^\ast)^2
\end{array}\right]\right.\\
&\left.+\left[4x_1^\ast~0\right]\left[\begin{array}{l}
(x_1-x_1^\ast)^3\\(x_2-x_2^\ast)^3
\end{array}\right]+\left[1~0\right]\left[\begin{array}{l}
(x_1-x_1^\ast)^4\\(x_2-x_2^\ast)^4
\end{array}\right]\right|\\
\le&\left\|\left[\begin{array}{l}
4(x_1^\ast)^3\\2x_2^\ast
\end{array}\right]\right\|\left\|x-x_\ast\right\|+\max\left\{6(x_1^\ast)^2,1\right\}\left\|x-x_\ast\right\|^2+\left|4x_1^\ast\right|\left\|x-x_\ast\right\|^3+\left\|x-x_\ast\right\|^4.
\end{align*}
By $\|x-x_\ast\|\le\left(c_1L\right)^{1/2}$ followed by the value of Lyapunov function $L(x)$, we obtain
\begin{align*}
\left|1-p(x)^{2c}\right|
&\le\left\{16(x_1^\ast)^6+4(x_2^\ast)^2\right\}^{1/2}\left(c_1L\right)^{1/2}+\max\left\{6(x_1^\ast)^2,1\right\}c_1L+\left|4x_1^\ast\right|\left(c_1L\right)^{3/2}+\left(c_1L\right)^{2}\\
&=:C_{n,\alpha,N}(L).
\end{align*}

Finally we obtain an upper bound of $t_{\max}$ as follows:
\begin{align*}
t_{\max}
&=t_N+\int_{\tau_N}^\infty\left(1-p(x(\tau))^{2c}\right)^{k}\left(1-\frac{2c-1}{2c}\left(1-p(x(\tau))^{2c}\right)\right)d\tau\\
&=t_N+\int_{\tau_N}^\infty\left(1-p(x(\tau))^{2c}\right)^{k}\left(\frac{1}{2c}+\frac{2c-1}{2c}p(x(\tau))^{2c}\right)d\tau\\
&\le t_N+\int_{\tau_N}^\infty\left|1-p(x(\tau))^{2c}\right|^{k}d\tau\\
&\le t_N +\frac{1}{c_{\tilde N}c_1} \int_0^{L(x(\tau_N))} \frac{C_{n,\alpha,N}(L)^{k}}{L}dL, 
\end{align*}
where we have used the estimate $\frac{dL}{d\tau}\le -c_{\tilde N}c_1L$ along the trajectory $\{x(\tau)\}$, which follows from the inequality of Lyapunov functions.
The positive constants $c_{\tilde N}, c_1$ are shown in \cite{TMSTMO}.
\par
\bigskip

Next we show an estimate of $|1-p(x)^{2c}|$ with compactifications of general type $\alpha=(\alpha_1,\dots,\alpha_n)$.
As in the previous case, let $x_\ast = (x^\ast_1, \cdots, x^\ast_n) \in \mathcal{E}$ and assume that a Lyapunov function $L(x)$ is validated in a vicinity of $x_\ast$. 
Then
\begin{align*}
\left|1-p(x)^{2c}\right|
&=\left|1-\sum_{i=1}^nx_i^{2\beta_i}\right|\\
&=\left|\sum_{i=1}^n\left(x_i^\ast\right)^{2\beta_i}-\sum_{i=1}^n\left(x_i-x_i^\ast+x_i^\ast\right)^{2\beta_i}\right|\\
&=\left|\sum_{i=1}^n\sum_{j=1}^{2\beta_i}\binom{2\beta_i}{j}\left(x_i-x_i^\ast\right)^{j}\left(x_i^\ast\right)^{2\beta_i-j}\right|\\
&=\left|\sum_{j=1}^{\max\left\{2\beta_i\right\}}v_j^T\left[\begin{array}{c}
(x_1-x_1^\ast)^j\\(x_2-x_2^\ast)^j\\\vdots\\(x_n-x_n^\ast)^j
\end{array}\right]\right|,
\end{align*}
where $v_j\in\mathbb{R}^n$ is the vector given by
\[
	\left(v_j\right)_i=\left\{\begin{array}{ll}
	\binom{2\beta_i}{j}\left(x_i^\ast\right)^{2\beta_i-j}&(j\le 2\beta_i),\\
	0&(j>2\beta_i).
	\end{array}\right.
\]
Thus we have
\begin{align*}
\left|1-p(x)^{2c}\right|
&=\left|\sum_{j=1}^{\max\left\{2\beta_i\right\}}v_j^T\left[\begin{array}{c}
(x_1-x_1^\ast)^j\\(x_2-x_2^\ast)^j\\\vdots\\(x_n-x_n^\ast)^j
\end{array}\right]\right|\\
&\le \|v_1\|\|x-x_\ast\|+\sum_{j=2}^{\max\left\{2\beta_i\right\}}\|v_j\|_\infty\|x-x_\ast\|^j\\
&\le \|v_1\|\left(c_1L\right)^{1/2}+\sum_{j=2}^{\max\left\{2\beta_i\right\}}\|v_j\|_\infty\left(c_1L\right)^{j/2}=:C_{n,\alpha,N}(L),
\end{align*}
where we have used $\|x-x_\ast\|\le\left(c_1L\right)^{1/2}$.
\par
\bigskip
If $k=1$, which is the case shown in Section \ref{section-ex3}, then an upper bound estimate of $t_{\max}$ is realized as follows, for example:
\begin{align*}
t_{\max}
&\le t_N +\frac{1}{c_{\tilde N}c_1} \int_0^{L(x(\tau_N))} \frac{C_{n,\alpha,N}(L)}{L}dL\\
&= t_N +\frac{1}{c_{\tilde N}} \int_0^{L(x(\tau_N))} \left\{\|v_1\|\left(c_1L\right)^{-1/2}+\sum_{j=2}^{\max\left\{2\beta_i\right\}}\|v_j\|_\infty\left(c_1L\right)^{j/2-1}\right\}dL\\
&= t_N +\frac{1}{c_{\tilde N}} \left\{2\|v_1\|c_1^{-1/2}L(x(\tau_N))^{1/2}+\sum_{j=2}^{\max\left\{2\beta_i\right\}}\frac{2}{j}\|v_j\|_\infty c_1^{j/2-1}L(x(\tau_N))^{j/2}\right\}.
\end{align*}

\bibliographystyle{plain}
\bibliography{qh_blow_up_2}

\end{document}